\documentclass[11pt, reqno]{amsart}

\usepackage[margin=1.1in]{geometry}

\usepackage{amsmath, amsfonts, amsthm, amssymb, multicol, mathtools, dsfont, mathrsfs}
\usepackage{graphicx}
\usepackage{float, hyperref}
\usepackage{scalerel,stackengine, subcaption}
\usepackage[usenames,dvipsnames,x11names]{xcolor}
\usepackage{enumitem}
\usepackage{pgfplots}
\usepgfplotslibrary{fillbetween}
\pgfplotsset{width=10cm,compat=1.9}
\usepackage[square,sort,comma,numbers]{natbib}
\makeatletter
\g@addto@macro\bfseries{\boldmath}
\makeatother

\makeatletter
\def\@setauthors{%
  \begingroup
  \def\thanks{\protect\thanks@warning}%
  \trivlist
  \centering\footnotesize \@topsep30\p@\relax
  \advance\@topsep by -\baselineskip
  \item\relax
  \author@andify\authors
  \def\\{\protect\linebreak}

  \normalsize\lowercase{\authors}%
  
	\ifx\@empty\contribs
  \else
    ,\penalty-3 \space \@setcontribs
    \@closetoccontribs
  \fi
  \endtrivlist
  \endgroup
}
\def\@settitle{\begin{center}
\LARGE\lowercase{\@title}
  \end{center}%
}
\makeatother
\newcommand{\authoremail}[1]{\email{\href{mailto:#1}{\color{lightblue}{#1}}}}
\newcommand{\authoraddress}[1]{\address{\normalfont{#1}}}

\numberwithin{equation}{section}

\newtheorem{thm}{Theorem}[section]

\newtheorem{cor}[thm]{Corollary}
\newtheorem{defn}[thm]{Definition}

\newtheorem{prop}[thm]{Proposition}
\newtheorem{example}[thm]{Example}

\renewcommand{\epsilon}{\varepsilon}
\newcommand{\eps}{\varepsilon}

\newcommand{\rd}{\mathbb{R}^d}

\renewcommand{\geq}{\geqslant}
\renewcommand{\leq}{\leqslant}
\newcommand{\ubd}{\overline{\dim}_{\textup{B}}}
\newcommand{\lbd}{\underline{\dim}_{\textup{B}}}
\newcommand{\hd}{\dim_{\textup{H}}}
\newcommand{\frd}{\dim_{\textup{Fr}}}

\newcommand{\bd}{\dim_{\textup{B}}}

\newcommand{\fs}{\dim^\theta_{\mathrm{F}}}
\newcommand{\as}{\dim^\theta_{\mathrm{A}}}
\newcommand{\fd}{\dim_{\mathrm{F}}}
\newcommand{\sd}{\dim_{\mathrm{S}}}

\newcommand{\J}{\mathcal{J}}

\makeatletter
\DeclareRobustCommand\widecheck[1]{{\mathpalette\@widecheck{#1}}}
\def\@widecheck#1#2{%
    \setbox\z@\hbox{\m@th$#1#2$}%
    \setbox\tw@\hbox{\m@th$#1%
       \widehat{%
          \vrule\@width\z@\@height\ht\z@
          \vrule\@height\z@\@width\wd\z@}$}%
    \dp\tw@-\ht\z@
    \@tempdima\ht\z@ \advance\@tempdima2\ht\tw@ \divide\@tempdima\thr@@
    \setbox\tw@\hbox{%
       \raise\@tempdima\hbox{\scalebox{1}[-1]{\lower\@tempdima\box
\tw@}}}%
    {\ooalign{\box\tw@ \cr \box\z@}}}
\makeatother

\stackMath
\newcommand\reallywidehat[1]{%
\savestack{\tmpbox}{\stretchto{%
  \scaleto{%
    \scalerel*[\widthof{\ensuremath{#1}}]{\kern.1pt\mathchar"0362\kern.1pt}%
    {\rule{0ex}{\textheight}}
  }{\textheight}%
}{2.4ex}}%
\stackon[-6.9pt]{#1}{\tmpbox}%
}
\usepackage{fancyhdr}
\definecolor{lightblue}{HTML}{2B77A4}
\colorlet{plotblue}{LightSkyBlue3!80}
\definecolor{darkred}{HTML}{9E0D0D}
\hypersetup{
	colorlinks=true,
	linkcolor=darkred,
	urlcolor=darkred,
	citecolor=lightblue
}
\title[Quantitative flatness]{Quantitative flatness and obstructions in Fourier analysis}

\author{Jonathan M. Fraser}
\authoraddress{J. M. Fraser, University of St Andrews, Scotland}
\authoremail{jmf32@st-andrews.ac.uk}
\thanks{JMF was financially supported by an EPSRC Open Fellowship (EP/Z533440/1) and a Leverhulme Trust Research Project Grant (RPG-2023-281).}

\begin{document}

\thispagestyle{empty}
\begin{abstract}
The Fourier restriction  problem asks when it is meaningful to restrict the Fourier transform of a function to the support of a given measure.  A  different (but related) problem is to determine when convolution with the measure improves the integrability of a function.  A third problem  is to estimate the Fourier dimension of the measure, which involves understanding uniform decay rates for the Fourier transform.  Positive results for any of these problems require a quantitative understanding of various geometric properties  of the measure, including (broad interpretations of) curvature and arithmetic resonance.  In this paper we first establish a unified framework for providing  negative results for all three  problems (that is, we provide explicit obstructions to a measure satisfying certain Fourier restriction, $L^p$-improving, or Fourier decay estimates) by quantifying flat parts of the measure in the spirit of the well-known Knapp examples from harmonic analysis.  Our main interest is in the application of these abstract results in various concrete settings where we use analytic and fractal geometric concepts to force `flatness'.  

Our  framework applies  generally and this allows us to unify and extend various parts of the literature.  Some representative applications  include:~(i) we bound the Fourier dimension of the surface measure on a compact $C^2$ surface of arbitrary dimension above by the smallest ambient rank of a point on the surface; (ii) we prove that the Fourier dimension of a smooth curve in $\rd$ is at most $4/(d+1)$ and so such curves cannot be Salem for $d \geq 4$  with analogous results for higher dimensional submanifolds; (iii) we provide mild and readily checkable conditions ensuring that an ergodic measure on a self-affine set is not Salem; (iv) we obtain an explicit upper bound for the Fourier dimension of the Patterson--Sullivan measure for parabolic Kleinian group actions; (v) we  establish  novel connections between Fourier restriction/decay and \emph{a priori} unrelated concepts in fractal geometry including the Assouad spectrum of projections and slices,  and  a strong form of tube-nullity.   We establish several auxiliary results along the way, including a precise characterisation of $L^2$-flattening in terms of the Fourier spectrum, and we consider numerous examples throughout the paper.
\end{abstract}
\maketitle

\tableofcontents

\section{Introduction and overview of paper}

Let $\mu$ be a finite compactly supported Borel measure on $\rd$.  The \emph{Fourier restriction or extension problem}   asks when it is meaningful to extend the Fourier transform of an arbitrary $L^{p}(\mu)$ function to a function in   $L^{q}(\rd)$.   This is certainly not always possible and the values of $p$ and $q$ for which such an extension does hold depend on $\mu$ in a subtle way. See Section \ref{knappproof} for a more detailed discussion of this problem. A different but related problem is the \emph{$L^p$-improving problem}, which is to determine when convolution with $\mu$ improves the integrability of all $L^p$ functions.   See Section \ref{sec:improving} for a more detailed discussion of this problem. Yet another distinct but related problem is the \emph{Fourier decay problem} which asks how fast the Fourier transform of $\mu$ decays.  This can be done in an $L^\infty$ sense, leading to the Fourier dimension which captures the optimal uniform decay rate, or in a suitable average sense which leads to the Fourier spectrum,  see Section \ref{sec:decay} for the  definitions.  The precise values of the Fourier dimension and spectrum also depend  on $\mu$ in a subtle way.

The Fourier restriction, $L^p$-improving,  and Fourier decay problems are  connected by the fact that they  are all inhibited by `flatness'.  That is, if $\mu$ puts a lot of mass in  a relatively flat part of space, then one would expect the available $L^p$--$L^q$ extension, $L^p$-improving,   and Fourier decay estimates to be weaker.  In the restriction literature, `flatness' has often been captured by Knapp style examples.  Our first goal is to formulate a general Knapp example and to precisely describe its consequences for general restriction/extension estimates, see Theorem \ref{knapp}. This follows the classical Knapp style argument but serves to unify various estimates from the literature. Our next goal is to establish a Knapp style obstruction in the $L^p$-improving problem which extends and simplifies  previous results, see Theorem \ref{thm:improving}.    Turning to Fourier decay,   we next aim to prove a novel Knapp style obstruction in the Fourier decay problem, that is, to bound the Fourier dimension or spectrum from above in terms of `quantitative flatness'.  This turns out to be quite subtle and we spend some time formulating various different approaches to this (see Corollary \ref{fourier}, and Theorems  \ref{fourier2}, \ref{fourier3}, and \ref{fourier4})  and discussing why a natural conjecture in this direction fails (see, in particular, Example \ref{nowithsmallg}). Next we precisely characterise $L^2$-flattening in terms of the Fourier spectrum (see Theorem \ref{flattening}) and this allows us to show that the strongest form of quantitative flattening we consider also prohibits $L^2$-flattening. It is worth noting here that the two distinct uses of `flattening' in the previous sentence refer to very different things  and are in fact going in opposite directions:~if a measure is too `flat', then its Fourier transform is not good at `flattening' other measures!

Once our abstract obstructions to Fourier restriction, $L^p$-improving, Fourier decay, and $L^2$-flattening are established, our goal shifts to applications of these results.  We spend the rest of the paper establishing quantitative flatness in various situations.  We start with some basic estimates, recovering a result of Mitsis which relates the Fourier dimension and Frostman dimension and a basic obstruction to restriction estimates due to Hambrook--{\L}aba (see Theorem \ref{basic}).  Next we consider the classical setting for Fourier restriction, that is, when $\mu$ is the surface measure on a  $C^2$ surface.  Here there is currently considerable interest in Fourier dimension estimates where it has recently been proved that the Fourier dimension of a smooth hypersurface of constant rank is equal to its rank, see papers of Zhu \cite{zhu1,zhu2}, and also \cite{kroon}.  We are able to consider arbitrary co-dimension (that is, beyond co-dimension 1) and also variable rank.  We prove that the Fourier dimension of the surface measure on a $C^2$ surface is bounded above by the smallest ambient rank of a point on the surface (see Theorem \ref{smoothsurface1}).  We also prove that the Fourier dimension of a $C^2$ surface itself is bounded above by the largest ambient rank of a point on the surface, provided this number is sufficiently small (see Theorem \ref{smoothsurface2}). Both of these estimates are best possible and resolve, or make progress, on a folklore conjecture about how the Fourier dimension depends on the rank of a surface.  In a complementary direction, we prove that if a submanifold of $\rd$ has sufficiently small dimension relative to the ambient spatial dimension $d$, then it cannot be Salem and we give a bound for the Fourier dimension which goes to zero as $d \to \infty$ (see Corollaries  \ref{surfacedim1},  \ref{surfacedim2} and \ref{surfacedim3}).   For example,  we prove that the Fourier dimension of a smooth curve in $\rd$ is at most $4/(d+1)$ and the Fourier dimension of the arclength measure is at most $4/(d+2)$.

The general nature of our flatness framework means it is amenable to various geometric observations which \emph{a priori} have nothing to do with Fourier decay, such as tube nullity (see Theorem \ref{tubenull}), the Assouad spectrum of slices (see Theorem \ref{slices}, and Theorem \ref{nonlinearslices} for nonlinear slices) and the Assouad spectrum of projections (see Theorem \ref{projections}).

In the final part of the paper, we consider explicit examples coming from the fractal geometry literature, including self-affine sets (see Theorem \ref{affine}), limit sets of Kleinian groups (see Theorem \ref{kleinian}), and Julia sets of rational maps (see Theorem \ref{julia}). We use geometric features of these examples to establish quantitative flatness which in turn yields obstructions to the Fourier restriction and decay problems.  More precisely, we use the non-conformality in the construction of self-affine sets to force flatness inside basic construction cylinders and, in the case of Kleinian limit sets and Julia sets, we use parabolicity to force flatness in a (flat) neighbourhood of the parabolic fixed points.  Estimating the Fourier dimension in these three classes is notoriously difficult (and generally open) and often any positive results are non-quantitative and give rise to mysterious  numerical values, see \cite{banajibaker,banajiyu, bourgaindyatlov,jordansahlsten,naud, leclerc}.  By contrast, the bounds we obtain are clean and depend only on familiar parameters such as the entropy, Lyapunov exponents, parabolic ranks, and petal numbers.  The bounds we obtain often allow us to rule out the possibility that the measures in question are Salem. Moreover, we are able to exhibit examples in each class where the Fourier dimension is arbitrarily small while the Hausdorff dimension remains bounded away from zero.  A far as we are aware, we provide the first non-trivial negative results for the Fourier restriction, $L^p$-improving, or Fourier decay problems   for these well-studied classes of fractal. 

We refer the reader to \cite{mattilaFourier} for a comprehensive treatment of Fourier analysis and its applications in fractal geometry, as well as the Fourier restriction and decay problems.  For more background on Fourier analysis in general, see \cite{grafakos, sogge, Ste93}, and for fractal geometry, see \cite{barany,Fal03, jon:book}.

\section{Quantitative flatness:~abstract obstructions in Fourier analysis}

\subsection{Key definitions and terminology}

Throughout the paper we consider  non-zero finite compactly supported Borel measures $\mu$ on $\rd$ and we are interested in their Fourier analytic properties.  Our aim is to use quantified information pertaining to how `flat' or `concentrated'  the measure is to obstruct or inhibit what might be perceived as `good' Fourier analytic behaviour.  In order to do this we introduce $\gamma$-flat measures, which use a parameter $\gamma$ to ensure that a disproportionately large amount of the measure is concentrated inside a (usually anisotropic) cuboid.  We say $Q \subseteq \rd$ is a \emph{cuboid} if
\[
Q=g\left(I_1 \times \cdots \times I_d \right)+z
\]
where $g \in SO(d)$ is an orthogonal transformation of $\rd$,  $z \in \rd$ is a translate, and $I_j$ are non-empty non-degenerate compact intervals in $\mathbb{R}$ centred at the origin. We refer to $z$ as the \emph{centre} of the cuboid $Q$.

Here and throughout, we use the notation $A\lesssim B$ to mean that there exists a uniform constant $c$ such that $A\leq cB$.  The constant $c$ is allowed to depend on the ambient spatial dimension $d$ and on the measure in question.  It can also depend on parameters which are fixed, e.g. the flatness parameter $\gamma$.  But it cannot depend on variables, such as an element in the sequence of cuboids used to witness flatness or on particular length scales used in the relevant proof.  We also write  $A\approx B$ if $A\lesssim B$ and $B\lesssim A$.  We write $\mathcal{L}^d$ for the $d$-dimensional Lebesgue measure on $\rd$ and it is useful to keep in mind the simple fact that
\[
\mathcal{L}^d(Q) = \prod_{j=1}^d |I_j|
\]
where $|I_j|>0$ is the length of the interval $I_j$.

We now introduce the first of two key definitions which will be used throughout the paper. 

\begin{defn} \label{quantflat}
Let $\mu$ be a finite compactly supported Borel measure in $\rd$ and $\gamma \geq 0$.  We say $\mu$ is \textbf{$\gamma$-flat}  if there exists a sequence  $Q_k$ of cuboids such that
\[
\mu(Q_k) \gtrsim \mathcal{L}^d(Q_k)^\gamma
\]
holds for     all $k$ and   $\mathcal{L}^d(Q_k) \to 0$ as $k \to \infty$. We say that the sequence of $Q_k$ \emph{witnesses} the $\gamma$-flatness of $\mu$.
\end{defn}

If the cuboids are isotropic (that is, they are in fact cubes), then being $\gamma$-flat does not really imply anything about flatness, but rather that the measure is concentrated in a small ball.  This information is precisely described by the Frostman and box dimensions of the measure, see Section \ref{dimreview} for precise definitions.  However, the isotropic case is the trivial case where we will find the least interesting consequences.  For this reason we chose to use the term `flat' rather than, for example, `concentrated' to reflect that the anisotropic case, which does imply `flatness', is the most important. 

It is worth keeping in mind that $\gamma$-flatness gets harder to satisfy as $\gamma$ decreases.  Indeed, an easy pigeonholing argument shows that \emph{all}  measures are $1$-flat and there do exist $0$-flat measures, as we will see below.  Therefore, the interesting range is $\gamma \in [0,1)$.

Often we will need a refinement of the notion of $\gamma$-flatness where we not only  pay attention to the mass of the cuboids, but also how the Lebesgue volume of the cuboid relates to the shortest sidelength.  In some sense this is also quantifying `how flat' the cuboids are.  This is captured by the following, which is our second key definition.

\begin{defn} \label{quantflat2}
Let $\mu$ be a finite compactly supported measure in $\rd$, $\gamma \geq 0$ and $\upsilon \in [1,d]$.  We say $\mu$ is \textbf{$(\gamma, \upsilon)$-flat} if $\mu$ is $\gamma$-flat with a sequence $Q_k$ of  witnessing cuboids with shortest sides denoted by
\[
\delta_k = \min_j |I_j^k|
\]
such that
\[
\mathcal{L}^d(Q_k) \gtrsim \delta_k^\upsilon
\]
holds for all $k$.
\end{defn}

In the above, $\upsilon = d$ corresponds to the least interesting isotropic case where the $Q_k$ are (up to constants)   cubes of volume $\delta_k^d$.  On the other hand,  $\upsilon=1$ applies to situations in the other extreme where the cuboids have sidelengths (up to constants)  $\delta_k \times 1 \times \cdots \times 1$.

To the best of our knowledge, the concept of being $\gamma$-flat (Definition \ref{quantflat}) is new in this level of generality, although it has appeared before in the context of measures on submanifolds.  In particular, the \emph{Oberlin condition} from \cite[(1)]{gressman} (see also \cite{oberlin}) is essentially the negation of being $\gamma$-flat.  This consideration also leads to Oberlin's \emph{affine dimension}, introduced in \cite{oberlinaffine}, which involves constructing an affine Hausdorff measure based on covers by cuboids.   Moreover, these quantified descriptions of flatness---along with the classical Knapp examples---are known to be related to Fourier restriction, see \cite{demeter, gressman, mattilaFourier, oberlin,oberlinaffine} and the references therein.  The concept of being $(\gamma,\upsilon)$-flat (Definition \ref{quantflat2}) appears to be genuinely  new, although it is so simple and natural that it may very well have been considered elsewhere.  

In the following subsections, we will deduce consequences of $\gamma$-flatness and $(\gamma,\upsilon)$-flatness concerning several well-known problems in Fourier analysis.  

\subsection{Fourier restriction:~a general Knapp example} \label{knappproof}

The  Fourier restriction problem asks when it is meaningful to restrict the Fourier transform of a function to a set of measure zero.  Particular instances of this problem have become some of the most famous and long-standing problems in harmonic analysis, going back to work by Stein (see \cite{Ste70, Tom75}), and have deep  connections with  geometric measure theory and  PDEs. Some of the most important examples include smooth co-dimension 1 manifolds such as the sphere, cone and hyperboloid.  Here and throughout, we   write  $p'$ and $q'$ for the H\"older conjugates of $p,q \in [1,\infty]$, defined by $\frac{1}{p} + \frac{1}{p'} = \frac{1}{q} + \frac{1}{q'}=1$. 

More precisely, the \emph{Fourier restriction problem} asks, given a measure $\mu$ supported on a Lebesgue null set in $\rd$, for which values of $q'\in[1,2)$ and  $p'\in[1,\infty]$ is  the Fourier transform  $L^{q'}(\rd) \to L^{p'}(\mu)$ bounded. That is, we say $\mu$ satisfies the $R(q'\to p')$ restriction estimate  if
\begin{equation}\label{eq:restriction}
  \| \widehat{f}\, \|_{L^{p'}(\mu)} \lesssim \| f \|_{{L^{q'}(\rd)}}
\end{equation}
holds for all $f\in L^{q'}(\rd)$. By $L^p$ duality, the $R(q'\to p')$ restriction estimate \eqref{eq:restriction} is equivalent to   the $E(p\to q)$  extension estimate 
\begin{equation}\label{eq:extension}
    \| \widehat{f\mu} \|_{L^{q}(\rd)} \lesssim \| f \|_{L^{p}(\mu)}
\end{equation}
 holding  for all $f\in L^{p}(\mu)$.  Note that we restrict our attention to $q>2$ because if $q=q'=2$ then we consider $f \in L^2(\rd)$ and the Fourier transform of such functions is only defined almost everywhere and so we cannot make sense of the restriction to a Lebesgue null set.  We will refer to the extension estimates throughout the paper, noting the above equivalence with restriction.  Our goal will be to impose restrictions on the pairs $(p,q)$ for which a given measure satisfies the    $E(p\to q)$  extension estimate \eqref{eq:extension}.

The classical `Knapp example' in restriction theory establishes a necessary condition for restriction estimates to hold for the surface measure on the  sphere. The idea is to choose a small `cap' which contains a relatively large amount of measure in a relatively flat piece of the surface in question.  Recall that genuinely flat sets have no Fourier decay, and therefore no Fourier restriction estimates are possible for them.  Knapp examples extend this idea by showing that `roughly flat' pieces provide  obstructions to obtaining good restriction and extension estimates.

First we establish an abstract Knapp obstruction which we will then apply to several examples.  This construction is a modification of the standard Knapp example but we are unaware of it appearing in the literature in this generality. The most general versions we are aware of are \cite[Example 1.8]{demeter} and \cite{oberlin} (see \cite{gressman})  but these are still in the context of   measures on smooth surfaces.  

\begin{thm} \label{knapp}
Let $\mu$ be a  $\gamma$-flat measure in $\rd$ for some $\gamma\in [0,1]$.  If the $E(p\to q)$ extension estimate 
\eqref{eq:extension} holds for $p \geq 1$ and $q \geq 2$, then
\[
\frac{p}{q(p-1)} \leq \gamma.
\]
In particular, if $p=2$, then $q \geq 2/\gamma$.
\end{thm}

\begin{proof}
 Let $Q_k$ be a sequence of cuboids witnessing the $\gamma$-flatness of $\mu$ and write 
\[
Q_k = g_k\left(I_1^k \times \cdots \times I_d^k \right)+z_k
\]
where $g_k \in SO(d)$ is an orthogonal transformation of $\rd$ and $z_k \in \rd$ is a translate and $I_j^k$ are non-degenerate non-empty closed intervals in $\mathbb{R}$ centred at the origin.  Then, by assumption,
\[
\mu(Q_k) \gtrsim \mathcal{L}^d(Q_k)^\gamma
\]
holds   for all $k$ and   $\mathcal{L}^d(Q_k) \to 0$ as $k \to \infty$.  Let $Q_k^*$ be the  origin centred dual of $Q_k$ defined by
\[
Q_k^*=g_k\left((I_1^k)^* \times \cdots \times (I_d^k)^* \right)
\]
where the $(I_j^k)^*$ are non-empty closed intervals in $\mathbb{R}$ centred at the origin with length given by
\[
|(I_j^k)^*| =  \frac{1}{100 \sqrt{d}} |I_j^k|^{-1}.
\]
One easily sees that 
\[
\mathcal{L}^d(Q_k^*) = |(I_1^k)^*| \cdots |(I_d^k)^*| \approx \frac{1}{|I_1^k| \cdots |I_d^k|} = \mathcal{L}^d(Q_k)^{-1}
\]
and that for all $x \in Q_k$ and $\xi \in Q_k^*$ 
\[
| \xi\cdot(x-z_k)  | \leq 1/100.
\]
Let $f_k = 1_{Q_k}$ be the indicator function on $Q_k$. Then
\[
\|f_k\|_{L^p(\mu)} = \mu(Q_k)^{\frac{1}{p}}
\]
and
\begin{align*}
  \| \widehat{f_k\mu} \|_{L^q(\rd)} &\geq \Bigg( \int_{Q_k^*} \Bigg| \int_{Q_k} e^{-2\pi i \xi\cdot x} \,d\mu(x) \Bigg|^q \,d\xi \Bigg)^{\frac{1}{q}}\\
  &= \Bigg( \int_{Q_k^*} \Bigg| \int_{Q_k} e^{-2\pi i \xi\cdot (x-z_k)} \,d\mu(x) \Bigg|^q \,d\xi \Bigg)^{\frac{1}{q}}\\
  &\gtrsim \Bigg( \int_{Q_k^*}   \mu(Q_k)   ^q\, d\xi \Bigg)^{\frac{1}{q}}\\
  &\approx \mu(Q_k)  \mathcal{L}^d(Q_k^*)^{\frac{1}{q}}.
\end{align*}
Therefore, if we assume the $E(p\to q)$ extension estimate 
\eqref{eq:extension} holds, combining the previous two estimates gives that
\begin{equation*}
\frac{\mu(Q_k) }{ \mathcal{L}^d(Q_k)^{\frac{1}{q}}} \approx \mu(Q_k)  \mathcal{L}^d(Q_k^*)^{\frac{1}{q}} \lesssim   \| \widehat{f_k\mu} \|_{L^q(\rd)} \lesssim \|f_k\|_{L^p(\mu)} =  \mu(Q_k)^{\frac{1}{p}}
\end{equation*}
and therefore
\begin{equation*}
\mathcal{L}^d(Q_k)^\gamma \lesssim \mu(Q_k) \lesssim \mathcal{L}^d(Q_k)^{\frac{p}{q(p-1)}}    
\end{equation*}
must hold for all $k$. We conclude that
\[
\frac{p}{q(p-1)} \leq \gamma,
\]
as required.
\end{proof}

An immediate corollary of the above and the observation that all measures are 1-flat is the following universal obstruction.

\begin{cor} \label{unirest}
Let $\mu$ be a finite non-zero compactly supported Borel measure on $\rd$.  If the $E(p\to q)$ extension estimate 
\eqref{eq:extension} holds for $p \geq 1$ and $q \geq 2$, then
\[
\frac{1}{q}+\frac{1}{p}\leq 1.
\]
\end{cor}

The obstruction obtained in Theorem \ref{knapp} gets stronger as $\gamma$ decreases. When $\gamma=1$ it does not provide any non-trivial information (since all measures are 1-flat) and when $\gamma=0$ it provides the strongest possible conclusion which is that there are no non-trivial restriction or extension estimates. In particular, it forces $q'=1$ and  the Fourier transform of an $L^1$ function is continuous   and so in $L^{p'}(\mu)$ for all $p'$, recalling that $\mu$ is compactly supported.  Therefore, 0-flatness forbids all but the trivial estimates which hold for all measures.

As a warm up to more complicated examples, we show how Theorem \ref{knapp} applies in some classical settings. These examples are, of course, not new.

\begin{example} \label{sphereeg}
Here we consider the classical Knapp example for the sphere $S^{d-1}$ equipped with the surface measure $\sigma^{d-1}$.  By taking $Q$ to be an optimal cuboid containing $B(x,\delta) \cap S^{d-1}$ for some $x \in S^{d-1}$ we note that $Q$ has one short side of length $\approx \delta^2$ and $(d-1)$ long sides of length $\approx \delta$.  Therefore,  
\[
\sigma^{d-1} (Q) \approx \delta^{d-1}
\]
and
\[
\mathcal{L}^d(Q) \approx \delta^{d+1} 
\]
and so $\sigma^{d-1}$ is $\gamma$-flat for $\gamma = \frac{d-1}{d+1}$ and Theorem \ref{knapp}   gives the (conjecturally sharp) bound
\[
\frac{p}{q(p-1)} \leq  \frac{d-1}{d+1}.
\]
We also see that $\sigma^{d-1}$ is $(\gamma,\upsilon)$-flat for $\gamma$ as above and $\upsilon=\frac{d+1}{2}$.
\end{example}

\begin{example} \label{coneeg}
Here we consider the classical Knapp example for the cone  
\[
C^{d-1}=\{(x,|x|) : x \in [0,1]^{d-1}\}
\]
 equipped with the surface measure $\mu$.  By taking $Q$ to be an optimal cuboid containing the $\delta$-neighbourhood (inside the cone) of the lift of a unit line through zero in $\mathbb{R}^{d-1}$ onto $C^{d-1}$ we get a cuboid with one long side of length $\approx 1$ and $(d-2)$ intermediate sides of length $\delta$ and one short side of length $\approx \delta^2$.  Then
\[
\mu(Q) \approx \delta^{-1} \, \delta^{d-1} = \delta^{d-2}
\]
and
\[
\mathcal{L}^d(Q) \approx \delta^{d-2} \delta^2 = \delta^{d} 
\]
and so $\mu$ is $\gamma$-flat for  $\gamma = \frac{d-2}{d}$  and Theorem \ref{knapp}    gives the (conjecturally sharp) bound
\[
\frac{p}{q(p-1)} \leq  \frac{d-2}{d}.
\]
We also see that $\mu$ is $(\gamma,\upsilon)$-flat for $\gamma$ as above and $\upsilon=\frac{d}{2}$.
\end{example}

\subsection{$L^p$-improving estimates} \label{sec:improving}

Another important problem in harmonic analysis is to determine certain $L^p$-improving properties for the convolution operator, see \cite{hare, oberlin, Ste76, taowright}. More precisely, we say $\mu$ satisfies the $I(p \to q)$ improving estimate  if
 \begin{equation} \label{improvingest}
  \|f \ast \mu \|_{L^q(\rd)} \lesssim \|f \|_{L^p(\rd)} 
 \end{equation}
holds  for all $f \in L^p(\rd)$.  That is,  $\mu$  `improves'  all $L^p$ functions to   $L^q$  functions when used to define  a convolution operator. Stein \cite{Ste76} initiated the study of such  $I(p \to q)$ improving estimates and proposed the problem of trying to characterise which $I(p \to q)$ improving estimates hold in terms of the `size' of $\mu$.  It is straightforward to see that all finite measures satisfy the $I(p \to p)$ improving estimate and so the challenge is to establish \eqref{improvingest} for $q>p$.

The fact that being $\gamma$-flat inhibits the $L^p$-improving properties of a measure $\mu$ appears in  work of Oberlin \cite{oberlin} where certain measures on hypersurfaces were considered, see also \cite{gressman}.  Oberlin also connected the problem to affine dimension (of sets) in \cite{oberlinaffine}, which is  itself directly  connected to $\gamma$-flatness.  In particular, Oberlin obtains the same numerology as we do below, see \cite[Corollary 3, and Theorem 2]{oberlin}, and also \cite[Proposition 2]{oberlinaffine}. As with Theorem \ref{knapp}, we are not aware of the following result appearing in the literature in this generality, but it is a generalisation of a well-known phenomenon.

\begin{thm} \label{thm:improving}
Let $\mu$ be a  $\gamma$-flat measure in $\rd$ for some $\gamma\in [0,1]$.  If the $I(p\to q)$ improving  estimate 
\eqref{improvingest} holds for $p \geq 1$ and $q \geq 1$, then
\[
\frac{1}{p}-\frac{1}{q}\leq \gamma.
\]
\end{thm}

\begin{proof}
 Let $Q_k$ be a sequence of cuboids witnessing the $\gamma$-flatness of $\mu$  with centres $z_k \in \rd$ and let $f_k=1_{Q_k'}$ be the indicator function on the cuboid
 \[
 Q_k':=2(Q_k-z_k)
 \]
 which is formed by translating  $Q_k$ to be centred at the origin and then enlarging it by a factor of 2.   First, note that
 \begin{equation} \label{improving1}
   \|f_k  \|_{L^p(\rd)} = \mathcal{L}^d(Q_k')^{1/p} \approx_d  \mathcal{L}^d(Q_k)^{1/p} .
 \end{equation}
On the other hand, by Jensen's inequality,
 \[
 \left(\frac{1}{\mathcal{L}^d(Q_k)} \int_{Q_k} f_k \ast \mu \, dx\right)^q \leq  \frac{1}{\mathcal{L}^d(Q_k)} \int_{Q_k} (f_k \ast \mu)^q \, dx  \leq \frac{1}{\mathcal{L}^d(Q_k)} \|f_k \ast \mu \|_{L^q(\rd)}^q
 \]
 and so
 \begin{align}
  \|f_k \ast \mu \|_{L^q(\rd)} &\geq   \mathcal{L}^d(Q_k)^{1/q-1}   \int_{Q_k} f_k \ast \mu \, dx \nonumber  \\
  &=  \mathcal{L}^d(Q_k)^{1/q-1}   \int_{Q_k} \mu(x-Q_k') \, dx \nonumber \\
  &\geq \mathcal{L}^d(Q_k)^{1/q}  \frac{1}{\mathcal{L}^d(Q_k)} \int_{Q_k} \mu(Q_k) \, dx \nonumber \\
  & =  \mathcal{L}^d(Q_k)^{1/q}    \mu(Q_k)   \label{improving2}.
 \end{align}
 In the penultimate line above we crucially used that, for all $x \in Q_k$, 
 \begin{equation} \label{contain}
 Q_k \subseteq x-Q_k'  
  \end{equation}
 and this is where we needed to consider the translated and enlarged cuboid $Q_k'$.  To justify \eqref{contain}, write $Q_k=Q_k^0+z_k$ where $Q_k^0$ is the corresponding origin centred cuboid.  Then we need to show that, for all $x_0 \in Q_k^0$, 
 \[
 Q_k^0+z_k \subseteq  x_0+z_k - 2Q_k^0 \   \Leftrightarrow  \  Q_k^0 \subseteq  x_0  - 2Q_k^0 
 \]
 and this is then an  elementary inclusion.
  
 Combining the estimates \eqref{improving1} and \eqref{improving2} and assuming that $\mu$ satisfies the $I(p \to q)$ improving estimate \eqref{improvingest}, we get
 \[
 \mathcal{L}^d(Q_k)^{1/q}    \mu(Q_k)  \leq   \|f_k \ast \mu \|_{L^q(\rd)} \lesssim \|f_k  \|_{L^p(\rd)} \approx \mathcal{L}^d(Q_k)^{1/p}
 \]
 and therefore
 \[
 \mu(Q_k) \lesssim \mathcal{L}^d(Q_k)^{1/p-1/q}.
 \]
 In particular, if $\mu$ is $\gamma$-flat, then
 \[
 1/p-1/q \leq \gamma
 \]
 as required.
 \end{proof}

As with Corollary \ref{unirest}, an immediate corollary of the above and the observation that all measures are 1-flat is the following universal obstruction.

\begin{cor} \label{uniimpr}
Let $\mu$ be a finite non-zero compactly supported Borel measure on $\rd$.  If the $I(p\to q)$ improving estimate 
\eqref{eq:extension} holds for $p \geq 1$ and $q \geq 1$, then
\[
\frac{1}{p}-\frac{1}{q}\leq 1.
\]
\end{cor}

Further, when $\mu$ satisfies the strongest form of flatness, then we see that no non-trivial  $L^p$-improving estimates are possible.

\begin{cor} \label{noimpro}
Suppose  $\mu$ on $\rd$  is   $\gamma$-flat for all $\gamma>0$.   Then $\mu$ does not satisfy the $I(p\to q)$ improving estimate 
\eqref{eq:extension} for $q>p$.
\end{cor}

\ 

\subsection{Fourier decay:~bounding the Fourier dimension and spectrum}

\subsubsection{The Fourier dimension and Fourier spectrum} \label{sec:decay}

In this section we recall the definitions of the Fourier dimension and spectrum and describe some of their basic properties. Decay of the Fourier transform is a fundamental property in Fourier analysis and describes many geometric and analytic features of the measure in question. The optimal decay rate is described by the \emph{Fourier dimension}  of $\mu$, defined by
\begin{equation*}
  \fd \mu = \sup\Big\{ s \in \mathbb{R}:   \sup_{\xi\in\rd}\big| \widehat{\mu}(\xi) \big|^2|\xi|^{s}<\infty \Big\}.
\end{equation*}
The \emph{Fourier dimension} of a  set  $X \subseteq \rd$ is then
\begin{equation*}
  \fd X = \sup\Big\{ s\in[0,d] : \exists \mu\text{ on }X \text{ such that  }  \sup_{\xi\in\rd}\big| \widehat{\mu}(\xi) \big|^2|\xi|^{s}<\infty \Big\}.
\end{equation*}
The \emph{Hausdorff dimension} of $X$, denoted by $\hd X$, may also be characterised by Fourier decay, but in an $L^2$ sense.  This correspondence  opens the door to many deep connections between fractal geometry and Fourier analysis.  In particular, 
\begin{equation*}
    \hd X = \sup\bigg\{ s\in[0,d] :\exists \mu\text{ on }X \text{ such that  }  \int_{\rd} \big| \widehat{\mu}(\xi) \big|^{2}|\xi|^{s - d} \,d\xi < \infty \bigg\}.
\end{equation*}
 The analogue for measures is the  \emph{Sobolev dimension} of $\mu$ which is defined by
\begin{equation*}
    \sd \mu  = \sup\bigg\{ s \in \mathbb{R} : \int_{\rd} \big| \widehat{\mu}(\xi) \big|^{2}|\xi|^{s - d} \,d\xi < \infty \bigg\}.
\end{equation*}
When it is less than $d$ (which is typical for us), the Sobolev dimension of a measure is also known as the energy dimension, the correlation dimension, or the $L^2$-dimension. We say a set $X$ is \emph{Salem} if $\hd X = \fd X$.  Such sets have `optimal Fourier analytic structure' and are useful in many applications.  As such, it is desirable to determine whether  specific examples are Salem or not. Similarly, we say a measure $\mu$ on $\rd$ is \emph{Salem} if $\hd \mu = \min\{\fd \mu,d\}$. If $\mu$ is supported on a Lebesgue null set, as most of the measures we consider are, then $\fd \mu \leq d$ and the definition of a Salem measure simplifies.  The \emph{Hausdorff dimension} of $\mu$ is defined by
\[
\hd \mu = \inf\{ \hd E: \mu(E)>0\}.
\]
Given that the Fourier dimension is about $L^\infty$ control of the Fourier transform and the Hausdorff and Sobolev dimensions are about $L^2$ control, it is natural to explore $L^p$ control for intermediate $p$ in an appropriate sense.  This is captured by the  Fourier spectrum, originally defined  in \cite{Fra24}, which   is a   family of dimensions that continuously interpolate between the   Fourier and Hausdorff dimensions for sets and the Fourier and Sobolev dimensions for measures.  First, define the $(s,\theta)$-energies for $s \in \mathbb{R}$ and $\theta\in(0,1]$ as
\begin{equation*}
    \J_{s,\theta} (\mu) = \bigg( \int_{\rd} \big| \widehat{\mu}(\xi) \big|^{\frac{2}{\theta}} |\xi|^{\frac{s}{\theta}-d} \,d\xi \bigg)^\theta,
\end{equation*}
and for $\theta = 0$ as
\begin{equation*}
    \J_{s,0}(\mu) = \sup_{\xi\in\rd}\big| \widehat{\mu}(\xi) \big|^2|\xi|^{s}.
\end{equation*}
Then the \emph{Fourier spectrum} of $\mu$ at $\theta\in[0,1]$ is given by
\begin{equation*}
    \fs \mu = \sup\{ s \in \mathbb{R} : \J_{s,\theta}(\mu)<\infty \},
\end{equation*}
and the \emph{Fourier spectrum} of $X$ at $\theta\in[0,1]$ is
\begin{equation*}
    \fs X = \sup\{ s\in[0,d] :\exists \mu\text{ on }X \text{ such that  }  \J_{s,\theta}(\mu) < \infty \}.
\end{equation*}
We may always assume the $\mu$ used to witness the Fourier dimension and spectrum of $X$ are compactly supported, even when $X$ is unbounded, see \cite[Lemma 1]{ekstrom} and \cite[Theorem 1.5]{Fra24}.  The dimensions mentioned above satisfy 
\[
0 \leq \fd X = \fd^0 X \leq \fs X\leq \fd^1 X = \hd X \leq d.
\]
Moreover, the Fourier spectrum $\fs X$ is non-decreasing and continuous for $\theta\in[0,1]$.  For measures, the Fourier spectrum is concave, continuous on $(0,1]$ and satisfies 
\[
0 \leq \fd \mu = \fd^0 \mu \leq \fs \mu \leq \fd^1 \mu =  \sd \mu.
\]
 Moreover, for compactly supported measures it is continuous on the closed interval $[0,1]$.  Importantly, it was proved in \cite{CFdO24a+} that if $\mu$ is compactly supported, then
\begin{equation} \label{bound}
\fs \mu \leq \fd \mu + d \theta.
\end{equation}
Despite its recent inception, the Fourier spectrum has already seen several applications  in both fractal geometry and harmonic analysis; see, for example, \cite{Fra24, pham} for applications to the distance set problem,  \cite{FdO24+} for applications to exceptional set estimates for orthogonal projections,    and \cite{CFdO24b+} where it was used to study the Fourier restriction problem (in the positive direction).

There are various enemies of Fourier decay.  For example, if a measure is supported on a hyperplane then it cannot decay at all in directions in the orthogonal complement of the hyperplane. On the other hand, the arithmetic resonance in the middle third Cantor set prevents any measure supported on the Cantor set from decaying at all.  Nevertheless, the Fourier spectrum captures more nuanced average decay in both of these settings.   The goal now is to use the quantitative flatness framework described in the previous section to inhibit Fourier decay.  This turns out to be more challenging than with the Fourier restriction problem, but we are able to make some progress, which will turn out to be useful in several specific settings later on.  Our main interest is in bounding the Fourier dimension from above, but we will often state our results in terms of the Fourier spectrum since these statements contain the Fourier dimension estimates (at $\theta=0$) and contain strictly more information whenever we can beat the estimate
\[
\fs \mu \leq \fd \mu +d \theta
\]  
which always holds.

\subsubsection{Bounding the Fourier dimension via Fourier restriction}

In this section we derive general upper bounds for the Fourier dimension and spectrum in terms of quantitative flatness, thus providing a Knapp style obstruction in the Fourier decay problem. 

The  Stein--Tomas restriction theorem ensures a range of $q$ for which the $E(2\to q)$ extension estimate   holds  in terms of the Fourier dimension and Frostman dimensions of a measure.  The \textit{Frostman dimension} of $\mu$  with support given by $X$ is defined by
\begin{align*}
\frd \mu = \sup \Big\{ s \geq 0 :  \exists C>0 \ , \ \forall \ r>0 \  ,  \ \forall x \in \rd \  ,  \  \mu(B(x,r)) \leq Cr^{s} \Big\} .
\end{align*}
 Therefore, since Theorem \ref{knapp} uses $\gamma$-flatness to inhibit extension estimates, we may bound the Fourier dimension from above by appealing to $\gamma$-flatness. 

\begin{cor} \label{fourier}
Let $\mu$ be a  $\gamma$-flat measure in $\rd$ for some $\gamma\in [0,1)$.   If $\alpha \in (0,d)$ is the Frostman dimension of $\mu$, then
\[
\fd \mu \leq \frac{2\gamma(d-\alpha) }{1-\gamma}.
\]
\end{cor}

\begin{proof}
Comparing the $p=2$ case of Theorem \ref{knapp} with the Stein--Tomas estimate we must have
\[
2+\frac{4(d-\alpha)}{\fd \mu} \geq \frac{2}{\gamma}
\]
and re-arranging gives the claim. 
\end{proof}

One drawback of Corollary \ref{fourier} is that it requires knowledge of the Frostman dimension of $\mu$.  We can use the general estimate $2\alpha \geq \fd \mu$ due to Mitsis \cite{Mit02} (see also Theorem \ref{basic}) to get the following Frostman independent bound. 

\begin{cor} \label{fouriernofrost}
If  $\mu$ is a  measure in $\rd$ supported on a Lebesgue null set which is $\gamma$-flat for some $\gamma\in [0,1)$, then
\[
\fd \mu \leq 2 \gamma d.
\]
\end{cor}

\begin{proof}
By combining the estimate of Mitsis and Corollary \ref{fourier} we get the bound
\[
\fd \mu \leq \min\left\{2 \alpha, \ \frac{2\gamma(d-\alpha) }{1-\gamma} \right\} \leq 2\gamma d
\]
where the final estimate comes from balancing the two terms in the minimum.
\end{proof}

In order for the estimate from the above corollary to be non-trivial the $\gamma$ needs to be rather small, certainly satisfying $\gamma<\hd \mu/(2d)$.

In \cite{CFdO24b+} the Stein--Tomas estimate was generalised to use information from the Fourier spectrum rather than only the Fourier dimension.  Following the same approach as Corollary \ref{fourier} we get an estimate for the Fourier spectrum.

\begin{cor} \label{fourierspectrum}
Let $\mu$ be a  $\gamma$-flat measure in $\rd$ for some $\gamma\in [0,1)$.  If $\alpha \in (0,d)$ is the Frostman dimension of $\mu$, then
\[
\fs \mu \leq \frac{2\gamma(d-\alpha) }{1-\gamma} + \theta \cdot \frac{\alpha - d \gamma}{1-\gamma}
\]
for all $0 \leq \theta \leq 2\gamma$.
\end{cor}

\begin{proof}
Comparing the $p=2$ case of Theorem \ref{knapp} with the restriction estimate from \cite{CFdO24b+}, we must have
\[
2+ \frac{2(d-\alpha)(2-\theta)}{\fs \mu-\alpha \theta} \geq \frac{2}{\gamma}
\]
for all $\theta \in (0,1)$ such that $\fs \mu > d \theta$ and re-arranging gives the claim. 
\end{proof}

One interesting feature of Corollary \ref{fourierspectrum} is that it is stronger than the estimate coming from Corollary \ref{fourier} combined with basic information about the Fourier spectrum, see \eqref{bound}.  In particular, 
\[
\fs \mu \leq \fd \mu + d \theta \leq  \frac{2\gamma(d-\alpha) }{1-\gamma} + d \theta
\]
but the right hand side is a strictly worse estimate than Corollary \ref{fourierspectrum}, recalling that we are interested in the  $\alpha<d$ regime.  

\subsubsection{Bounding the Fourier dimension using flatness directly}

If a measure is $\gamma$-flat, especially with anisotropic witnessing cuboids and $\gamma<1$, then $\mu$ must have large amounts of mass concentrated in relatively flat regions. This is an enemy of Fourier decay and so one might hope to use this information directly to bound the Fourier dimension from above, that is, without going via Fourier restriction.  It turns out that for this we need the refined concept of $(\gamma,\upsilon)$-flatness, but even in this setting we only have partial success.  We are able to prove some estimates in certain restricted situations and we explain why something stronger cannot hold.  The estimates we obtain here should be compared with those coming from restriction theory in Corollary \ref{fourier2}.  Sometimes one approach is better, sometimes the other is better, as we shall see later. 

We begin with a heuristic discussion which is meant to hint at what one might hope to get from   $(\gamma,\upsilon)$-flatness. To this end, let  $\mu$ be a  $(\gamma,\upsilon)$-flat measure   and $Q_k$ be a sequence of witnessing cuboids. To simplify this heuristic discussion, assume all of the $Q_k$ are centred at the origin. We will revisit the heuristic in full generality when we prove Theorem \ref{fourier4} later. Let
\[
\delta_k = \min_j |I_j^k|
\]
be the length of the shortest side of $Q_k$ and note   that $\delta_k \to 0$ as $k \to \infty$. We also have the estimates 
\[
\mathcal{L}^d(Q_k) \gtrsim \delta_k^\upsilon
\]
and
\[
\mu(Q_k) \gtrsim \mathcal{L}^d(Q_k)^\gamma
\]
by assumption.   Write $Q_k^*$ for the dual cuboid as in the proof of Theorem \ref{knapp}.  Consider the decomposition
\[
\mu = \mu|_{Q_k^c} + \mu|_{Q_k}  
\]
and observe that for $\xi \in Q_k^*$
\begin{align}\label{vgoodbound}
\widehat{\mu|_{Q_k} }(\xi)   = \int_{Q_k} e^{-2\pi i \xi\cdot x} \,d\mu(x) \in B\left( \mu(Q_k),  \mu(Q_k)/10\right).
\end{align}
In particular, for $\xi \in Q_k^*$
\begin{equation*} 
| \widehat{\mu|_{Q_k} }(\xi) | \approx \mu(Q_k) \gtrsim \delta_k^{\upsilon \gamma}.
\end{equation*}
Since the diameter of $Q_k^*$ is $\approx \delta_k^{-1}$, this suggests a bound  $\fd \mu \leq 2 \upsilon \gamma$ might be in play. Let
\[
 A_k^* \subseteq  Q_k^* \setminus B(0,\delta_k^{-1}/1000)
\]
be a cuboid with one side of length $\approx  \delta_k^{-1}$ and the other sides  of  lengths $\approx 1$. We are searching for $\xi$ with $|\xi| \approx \delta_k^{-1}$ for which
\[
\left| \widehat{\mu  } (\xi)\right| \gtrsim \mu(Q_k)
\]
and this is guaranteed if we can find $\xi \in A_k^* $ satisfying 
\[
\left| \widehat{\mu }(\xi) \right|  \gtrsim  \mu(Q_k).
\]
Indeed, if we have such $\xi \in A_k^* $, then 
\[
\left| \widehat{\mu  } (\xi)\right| \gtrsim \mu(Q_k) \gtrsim \mathcal{L}^d(Q_k)^\gamma  \gtrsim \delta_k^{\upsilon \gamma}\approx |\xi|^{-\frac{2\upsilon \gamma}{2}}
\]
and the bound 
\[
\fd \mu \leq 2\upsilon \gamma
\]
would follow.  Therefore,   suppose for all $\xi \in A_k^*$, 
\[
\left| \widehat{\mu}(\xi) \right|  \leq  \mu(Q_k)/10000.
\]
For this to be true, recalling \eqref{vgoodbound}, we require that for all $\xi \in A_k^*$,
\[
\widehat{\mu|_{Q_k^c} }(\xi) \in B\left(-\mu(Q_k), \mu(Q_k)/5\right)
\]
which seems like a very strong requirement.  Indeed, it is asking for the near-constancy of $\widehat{\mu|_{Q_k^c} }(\xi) $ for a very large set of $\xi$.   Any substantial  oscillation in $\widehat{\mu|_{Q_k^c} }$ (for example, a sign change) over this range of frequencies would kill this possibility immediately.  

However, despite the above observations and heuristics, the desired bound
\[
\fd \mu \leq 2 \upsilon \gamma
\]
is not true in general---not even close! 

\begin{example} \label{nosf}
Let $\mu = \phi \mathcal{L}^d$ where $\phi$ is a Schwartz density supported in $[0,1]^d$.  Then we know \emph{a priori} that $\mu$ enjoys rapid Fourier decay.  That is, for all $N \geq 1$,
\[
|\widehat{\mu} (\xi)| \lesssim_N |\xi|^{-N}.
\]
However, by considering $Q$ with dimensions $\delta \times \delta^\eps \times \dots \times \delta^\eps$ for $\eps>0$ we can obtain
\[
\mathcal{L}^d(Q) =\delta^{1+(d-1)\eps} 
\]
and
\[
\mu(Q) \approx  \mathcal{L}^d(Q)
\]
and so we may choose $\upsilon =1+(d-1)\eps$ and  $ \gamma = 1$. However, the upper bound 
\[
\fd \mu \leq 2 \upsilon \gamma = 2(1+(d-1)\eps) \to 2
\]
as $\eps \to 0$ does \emph{not} hold.
\end{example}

In Example \ref{nosf} we could have taken $\eps=0$ but we included the $\eps>0$ to ensure that the diameters of the cuboids tended to zero.  This is not necessary of course.  In fact the hoped for bound can fail much more generally as the next example shows.

\begin{example} \label{nog1}
Consider \emph{any} finite  compactly supported $\mu$ on $\rd$ and decompose $\mu$ into $\approx \delta^{-1}$ many parallel $\delta$ slabs (neighbourhoods of parallel hyperplanes).  Then there must be one of these slabs (which we take as $Q$) for which 
\[
\mu(Q) \gtrsim \delta
\]
and therefore the bound we hoped for above would give
\[
\fd \mu \leq 2
\]
which is clearly false in many cases (especially with $d \geq 3$).  
\end{example}

At this point the reader may be rightly suspicious of the previous two examples because they both had $\gamma=1$.  Recall that in Corollary \ref{fourier2} we necessarily needed $\gamma<1$ to get something non-trivial and so perhaps adding the assumption $\gamma<1$ can save the desired upper bound $2 \upsilon \gamma$.  Next we show that, unfortunately, this is not the case. 

\begin{example} \label{nowithsmallg}
Let $d \geq 4$, $0<\eps<\frac{d-3}{2(d-1)}$ and $E \subseteq \rd$ be a Salem set of dimension $t$ satisfying
\[
\frac{2}{1-2\eps}<t<d-1
\]
and
\[
\ubd E =t.
\]
Note that the constraint on $\eps$ ensures that such $t$ exist. Moreover, suppose $\mu$ is a measure on $E$ which witnesses the Fourier dimension precisely, that is, $\fd \mu = t$. It is straightforward to construct such $E$ and $\mu$, for example, by taking the pushforward of a suitable self-similar set and measure under Brownian motion (which is used to raise the Fourier dimension, see \cite{Kahane}).  Let $\delta>0$ be  small and consider the projection of $E$ onto a $(d-1)$-dimensional subspace.   This projection has upper box dimension at most $t$ and so we may cover it by $\lesssim \delta^{-t' \eps}$ many $\delta^\eps$ cubes for $t'>t$.  The preimage of each of these cubes under the projection is a tube which may be covered by $\lesssim \delta^{-1}$ many $\delta^\eps \times \cdots \times  \delta^\eps \times \delta$ cuboids and, therefore, by pigeonholing,  there exists such a cuboid  $Q$ with 
\[
\mu(Q) \gtrsim \delta^{t' \eps+1}
\]
and
\[
\mathcal{L}^d(Q) \approx \delta^{(d-1) \eps+1}.
\]
Therefore, $\mu$ is $(\gamma,\upsilon)$-flat for all
\[
1>\gamma > \frac{\eps t + 1}{ \eps(d-1)+1}
\]
and
\[
\upsilon=\eps(d-1)+1.
\]
Therefore, the heuristic upper bound 
\[
\fd \mu \leq 2 \upsilon \gamma = 2(\eps t +1)
\]
does \textbf{not} hold for all sufficiently small $\eps$, since $\fd \mu \geq t>2$.  By considering an elementary optimisation problem (in $t$ and $\eps)$ to make $\gamma$ as small as possible we find that we can take $\gamma$ arbitrarily close to 
\[
\frac{8(d-1)}{8(d-1)+(d-3)^2}
\]
by choosing $t$ as small as allowed by the constraint above and choosing $\eps = \frac{d-3}{4(d-1)}$. This can then be made arbitrarily small by taking $d$ large.  
\end{example}

The above example is disappointing but, perhaps surprisingly, we can recover the bound $\fd \mu \leq 2 \upsilon \gamma$ in some fairly general situations.  First, provided the product $\upsilon \gamma$ is strictly less than 1, we can recover the desired bound. 

\begin{thm} \label{fourier2-}
Let $\mu$ be a  $(\gamma, \upsilon)$-flat measure in $\rd$ such that  $\upsilon\gamma<1$.  Then
\[
\fd \mu \leq 2 \upsilon \gamma.
\] 
\end{thm}

The previous result  follows from the more general result below.

\begin{thm} \label{fourier2}
Let $\mu$ be a  $(\gamma, \upsilon)$-flat measure in $\rd$ and  suppose the sequence of witnessing cuboids are such that the shortest  $m \geq 1$ sides are all of comparable length up to uniform constants.   If $\upsilon\gamma<m$, then
\[
\fd \mu \leq 2 \upsilon \gamma.
\] 
\end{thm}
\begin{proof}
Let $\pi_k$ be the projection of $\mathbb{R}^d$ onto the $m$-dimensional subspace spanned by the shortest $m$ sides of $Q_k$.  Then $\pi_k( Q_k)$ is contained in a ball of radius $\approx \delta_k$.  By a result of  Mitsis \cite{Mit02} (which we have mentioned already)   the Frostman dimension of $\pi_k(\mu)$ satisfies
\[
\frd \pi_k(\mu) \geq \min\left\{\frac{\fd \pi_k(\mu)}{2}, m\right\}.
\]
Furthermore, by considering directions within the subspace we are projecting on to, it is clear that the Fourier dimension cannot decrease under projection, that is, 
\[
\fd \pi_k(\mu) \geq \fd \mu.
\]
Therefore, we may conclude that 
\[
(\pi_k \mu) \big(\pi_k( Q_k) \big) \lesssim  \delta_k^{\min\{s/2 , m\}}
\]
for any $s < \fd \mu$.  In particular, by inspecting  Mitsis' proof \cite{Mit02}, the implicit constants above are \emph{independent of $k$} (that is, independent of  the choice of projection, which is determined by the orientation of $Q_k$ and therefore by $k$). On the other hand, since $\mu$ is $(\gamma, \upsilon)$-flat, we also have
\[
\delta_k^{\upsilon \gamma} \lesssim \mu(Q_k) \leq (\pi_k \mu) \big(\pi_k( Q_k) \big)
\]
and so, since $\upsilon\gamma<m$, we are forced to have
\[
\upsilon\gamma\geq s/2
\]
and the result follows.
\end{proof}

It is  interesting to compare Example \ref{nowithsmallg} and Theorem \ref{fourier2}.  In Example \ref{nowithsmallg}, $\upsilon \gamma > 1+\eps t$ and this can be made arbitrarily close to 1.  But for Theorem \ref{fourier2} to apply with cuboids with the anisotropy present in Example \ref{nowithsmallg} we would need $\upsilon \gamma <1$.

One of the main limitations of Theorem \ref{fourier2} is that we are forced to assume some isotropy in the cuboid $Q_k$, that is, we are forced to have $m$ `shortest sides'.  This is needed to force the projection of $Q_k$ into a ball.  On the other hand, if we could control the measure of an anisotropic ball (e.g. an ellipsoid or cuboid) in terms of the Fourier dimension, then this problem would disappear.  Unfortunately, we cannot do this in a sharp enough way to get anything useful.  The best one can do directly is to cover the anisotropic ball with smaller isotropic balls but this is quite inefficient in general.  We leave the reader to formulate estimates coming from this approach. On the other hand, if the Fourier dimension is so big that   the projection is forced to have a continuous density, then the measure of anisotropic balls \emph{can} be controlled in a precise and useful way.  This leads to the following result.

\begin{thm} \label{fourier3}
Let $\mu$ be a finite compactly supported measure on $\rd$ and suppose there exists a sequence of cuboids $Q_k$ such that
\[
 \frac{\mu(Q_k)}{\delta_k^1 \delta_k^2 \cdots \delta_k^m} \to \infty
 \]
 where $\delta_k^1, \cdots , \delta_k^m$ are the $m$ shortest sides of $Q_k$ (which we do not assume to be comparable in length). Then $\fd \mu \leq 2m$.
\end{thm}

\begin{proof}
Let $\pi_k$ be projection of $\rd$ onto the $m$-dimension subspace spanned by  the $m$ shortest sides of $Q_k$. Suppose  $\fd \mu>2m$.  Then $\fd (\pi_k \mu) > 2m$ and so   the measure $\pi_k \mu$ is absolutely continuous with a continuous and  bounded density, see \cite[Theorem 3.4]{mattilaFourier}.  Therefore,
\[
(\pi_k \mu) \big(\pi_k( Q_k) \big) \lesssim  \mathcal{L}^m(\pi_k(Q_k)) =  \delta_k^1 \delta_k^2 \cdots \delta_k^m.
\]
 But we also have
\[
 \mu(Q_k) \leq (\pi_k \mu) \big(\pi_k( Q_k) \big)
\]
which  contradicts the assumption in the theorem.  Therefore $\fd \mu \leq 2m$ as required.
\end{proof}

It is clear from the formula in  Corollary \ref{fourier} that there is only non-trivial information when $\gamma<1$ but it is less immediately obvious why $\gamma<1$ is needed in  Theorems \ref{fourier2-} and \ref{fourier2}.  However, in (the more general) Theorem \ref{fourier2} the partial isotropy forces $\upsilon \geq m$ and since we require $\upsilon \gamma < m$ we must have $\gamma<1$.  In general we cannot expect any useful bounds to hold if $\gamma \geq 1$.  This can be seen by the Examples \ref{nosf} and \ref{nog1} above.

We conclude this section with another situation where we can obtain the desired upper bound.  The drawback this time is that we require an awkward integrability assumption \eqref{decaycondy} which   says that the measure decays away from the cubes witnessing the flatness. The aim was to   give a non-Fourier analytic condition.

\begin{thm} \label{fourier4}
Let $\mu$ be a $(\gamma, \upsilon)$-flat measure in  $\rd$ with $\gamma<1$ and let $Q_k$ be a  sequence of witnessing cuboids.  Assume the following uniform decay condition: 
\begin{equation} \label{decaycondy}
\sup_k   \int_{x \in    Q_k^c}  \prod_{j=1}^d  |x^j-z^j_k|^{-1}   \, d \mu(x) <\infty
 \end{equation}
 where $z_k=(z_k^1, \dots, z_k^d)$ is the centre of $Q_k$.  Then $\fd \mu \leq 2\upsilon \gamma$.
\end{thm}

\begin{proof} 
Let
\[
\delta_k = \min_j |I_j^k|
\]
be the length of the shortest side of $Q_k$ and recall   that $\delta_k \to 0$ as $k \to \infty$. We also have the estimates 
\[
\mathcal{L}^d(Q_k) \gtrsim \delta_k^\upsilon
\]
and
\[
\mu(Q_k) \gtrsim \mathcal{L}^d(Q_k)^\gamma
\]
by assumption.   Write $Q_k^*$ for the dual cuboid as in the proof of Theorem \ref{knapp}.  Consider the decomposition
\[
\mu = \mu|_{Q_k^c} + \mu|_{Q_k}  
\]
and recall that for all $x \in Q_k$ and $\xi \in Q_k^*$
\[
|\xi \cdot (x-z_k)| \leq 1/100.
\]
Therefore,
\begin{align*}
e^{2\pi i \xi\cdot z_k} \widehat{\mu|_{Q_k} }(\xi)   = \int_{Q_k} e^{-2\pi i \xi\cdot (x-z_k)} \,d\mu(x) \in B\left( \mu(Q_k),  \mu(Q_k)/10\right) \subseteq \mathbb{C}
\end{align*}
We are searching for $\xi$ with $|\xi| \approx \delta_k^{-1}$ for which
\[
\left| \widehat{\mu  } (\xi)\right| \gtrsim \mu(Q_k)
\]
and this is guaranteed if we can find 
\[
\xi \in A_k^* : = Q_k^* \setminus B(0,\delta_k^{-1}/1000)
\]
 satisfying 
\[
\left| \widehat{\mu }(\xi) \right|  \gtrsim  \mu(Q_k).
\]
Therefore, in order to reach a contradiction, suppose for all $\xi \in A_k^*$, 
\[
\left| \widehat{\mu}(\xi) \right|  \leq c \mu(Q_k)
\]
for some fixed $c$ chosen sufficiently small.  For this to be true, we require that for all $\xi \in A_k^*$,
\[
 e^{2\pi i \xi\cdot z_k}\widehat{\mu|_{Q_k^c} }(\xi) \in B\left(-\mu(Q_k), \mu(Q_k)/5\right)
\]
Therefore, averaging over  $\xi \in A_k^*$,
\begin{equation} \label{contra}
\mu(Q_k) \approx  \frac{1}{\mathcal{L}^d(A_k^*)} \left| \int_{\xi \in A_k^*}   e^{2\pi i \xi\cdot z_k}\widehat{\mu|_{Q_k^c} }(\xi)  \, d \xi \right|  \approx \mathcal{L}^d(Q_k)\left| \int_{\xi \in A_k^*}  e^{2\pi i \xi\cdot z_k} \widehat{\mu|_{Q_k^c} }(\xi)  \, d \xi\right| .
\end{equation}
Then
\begin{align*}
\mathcal{L}^d(Q_k)^{\gamma-1} &\lesssim  \mu(Q_k) \mathcal{L}^d(Q_k)^{-1} \qquad \text{(by assumption)}  \\ 
&\approx  \left| \int_{\xi \in A_k^*} e^{2\pi i \xi\cdot z_k} \int_{x    \in Q_k^c}  e^{-2 \pi i x \cdot \xi} \, d \mu(x) \, d \xi\right| \qquad \text{(by \eqref{contra})}  \\
&=  \left| \ \int_{x    \in Q_k^c} \int_{\xi \in A_k^*} e^{2\pi i \xi\cdot z_k} e^{-2 \pi i x \cdot \xi} \, d \xi \, d \mu(x)\right| \qquad \text{(by Fubini's theorem)}  \\
&\leq     \int_{x \in    Q_k^c}\left| \int_{\xi \in A_k^*} e^{-2 \pi i (x-z_k) \cdot \xi} \, d \xi \right|\, d \mu(x) \\
&\lesssim    \int_{x \in    Q_k^c}  \prod_{j=1}^d  |x^j-z^j_k|^{-1}   \, d \mu(x) \\
&\lesssim 1 
\end{align*}
by \eqref{decaycondy}. This is a contradiction since we assume $\gamma<1$.   Therefore, there exists $\xi$ with $|\xi| \approx \delta_k^{-1}$ for which
\[
\left| \widehat{\mu  } (\xi)\right| \gtrsim \mu(Q_k) \gtrsim \mathcal{L}^d(Q_k)^\gamma  \gtrsim \delta_k^{\upsilon \gamma}\approx |\xi|^{-\frac{2\upsilon \gamma}{2}}.
\]
It follows that 
\[
\fd \mu \leq 2\upsilon \gamma
\]
as required.
\end{proof}

\subsection{$L^2$-flattening}

$L^2$-flattening is an important property which a measure can satisfy which has many useful  applications, see \cite{algomkhalil, algomorponen,bks, khalil}.  Roughly speaking, the idea is  that convolution with an $L^2$-flattening measure quantitatively improves the  $L^2$ dimension, see \cite{rossi}. Importantly, $L^2$-flattening  does not require (but is implied by) positive Fourier dimension ($L^\infty$ decay).

We say a compactly supported Borel measure $\mu$ on $\rd$ is \emph{$L^2$-flattening} if for all $\eps>0$, there exists $p>1$ such that for all $R>0$
\[
\| \widehat{\mu} \|_{L^p(B(0,R))}^p = O_\eps(R^\eps).
\]
The following observation   shows that $L^2$-flattening is completely characterised by the Fourier spectrum.  Later we will use this characterisation to connect $L^2$-flattening to quantitative flatness.   Here we write $A(\theta) \sim B(\theta)$ as $\theta \to 0$  if $A(\theta)/B(\theta) \to 1$ as $\theta \to 0$. 

\begin{thm} \label{flatteningspectrum}
A compactly supported Borel measure $\mu$ on $\rd$ is \emph{$L^2$-flattening} if and only if either $\fd \mu >0$ or if 
\begin{equation} \label{der}
 \fs \mu \sim d \theta 
\end{equation}
as $\theta \to 0$.
\end{thm}

\begin{proof}
First assume that $\fd \mu >0$ in which case, for some $s>0$,  
\[
\big| \widehat{\mu}(\xi) \big| \lesssim |\xi|^{-s/2}.
\]
But then 
 \[
\| \widehat{\mu} \|_{L^p(B(0,R))}^p \lesssim  \int_{|\xi \leq R} (1+|\xi|)^{-ps/2} \, d \xi \leq  \int_{\xi\in \rd} (1+|\xi|)^{-ps/2} \, d\xi < \infty 
\]
for $p>2d/s$ and so $\mu$ is $L^2$-flattening.  Now suppose  $\fd \mu = 0$ in which case $\fs \mu \leq d \theta$  for all $\theta \in [0,1]$ by \eqref{bound} and we can appeal to \cite[Theorems 4.1 and 4.2]{Fra24} to get the following alternative characterisation of the Fourier spectrum.  Indeed, 
\[
\fs \mu = \sup\left\{ \alpha \in [0, d \theta ] : \| \widehat{\mu} \|_{L^{2/\theta}(B(0,R))}^{2 /\theta} \lesssim R^{d-\alpha/\theta} \right\}.
\] 
Therefore, recalling \eqref{bound} and that the Fourier spectrum of a measure is concave,  \eqref{der} implies that  for all $\eps>0$ there exists $\theta>0$ such that $\fs \mu > (d-\eps)\theta$.  Moreover, $\fs \mu > (d-\eps)\theta$ implies that
\[
\| \widehat{\mu} \|_{L^{2/\theta}(B(0,R))}^{2 /\theta} \lesssim R^{d-(d-\eps)\theta/\theta} = R^\eps
\]
and so $\mu$ is $L^2$-flattening, choosing $p(\eps) = 2/\theta>1$.

In the other direction, suppose $\mu$ is $L^2$-flattening.  Then, for all $\eps>0$, there exists $p>1$ such that for all $R>0$
\[
\| \widehat{\mu} \|_{L^p(B(0,R))}^p = O_\eps(R^\eps).
\]
But then
\[
\| \widehat{\mu} \|_{L^{2p}(B(0,R))}^p  \leq \| \widehat{\mu} \|_{L^p(B(0,R))}^p = O_\eps(R^\eps)
\]
and
\[
\| \widehat{\mu} \|_{L^{2p}(B(0,R))}^{2p}  \lesssim R^{2\eps}.
\]
Setting $\theta=1/p$, which is now safely in $[0,1]$, we get
\[
\fs \mu \geq (d-2\eps)\theta. 
\]
It follows that $\fs \mu \sim d \theta$ or that $\fd \mu >0$.  This completes the proof.
\end{proof}

This characterisation perhaps suggests a more thorough analysis of similar flattening properties.  For example, do  the higher order terms in $\fs \mu$ as $\theta \to 0$   play a role?  Moreover, when $\mu$ is not $L^2$-flattening, concavity of the Fourier spectrum implies that 
\[
\fs \mu \geq \theta \sd \mu
\]
holds for all $\theta$.  Therefore, one might ask  whether  an intermediate asymptotic, such as
\[
\fs \mu \sim u \theta 
\] 
as $\theta \to 0$ for some $u \in (\sd \mu , d)$, is useful or whether  we always need  $u=d$. Further, we wonder if there are interesting classes of measures which satisfy such  an intermediate asymptotic.

Since $L^2$-flattening occurs whenever $\fd \mu >0$, we can only inhibit $L^2$-flattening in the very special case when we can show $\fd \mu =0$.  But even in this case we have to do more and show that the Fourier spectrum is not asymptotically as large as possible (that is, does not satisfy $\fs \mu \sim d\theta$).  Perhaps surprisingly, we achieve the latter whenever we achieve the former. We also achieve  this in a uniform and quantitative sense by proving that the Fourier spectrum is bounded above by $(d-1/2)\, \theta$.

\begin{thm} \label{flattening}
Let $\mu$ be a measure on $\rd$ which is $\gamma$-flat for all $\gamma>0$. Then 
\[
\fs \mu \leq  (d-1/2) \, \theta 
\]
for all $0 \leq \theta \leq 1$.  In particular, the Frostman and Sobolev dimensions of $\mu$ are at most $d-1/2$ and  $\mu$ is not $L^2$-flattening.
\end{thm}

\begin{proof}
The strategy here is to contradict    Corollary  \ref{fourierspectrum} by convolving a measure with arbitrarily strong flatness but large Fourier spectrum with a   Salem  measure of small dimension. Suppose there exists $\theta_0>0$ such that $\dim_\mathrm{F}^{\theta_0} \mu \geq \kappa \theta_0$ for some $\kappa  >  d-1/2$. Then, by concavity of the Fourier spectrum,  $\fs \mu \geq \kappa \theta $ for all $\theta \in [0,\theta_0]$.  Let $\gamma>0$,   $\eps \in (0,2\theta_0)$ and let $\nu$ be a compactly supported Borel probability measure  with $\fd \nu = \eps$ which is Salem.  Then, by \cite[Theorem 6.1]{Fra24},
\[
\fs (\mu \ast \nu) \geq \fs \mu + \fd \nu \geq \kappa \theta  + \eps
\]
for all  $\theta \in [0,\theta_0]$. By assumption, we may find a cuboid $Q$ such that
\[
\mu(Q) \gtrsim \mathcal{L}^d(Q)^\gamma
\]
and let  $\delta$ be the length of the shortest side of $Q$.  Since we chose  $\nu$ to have Hausdorff (and so Frostman) dimension $\eps$ we can find a $\delta$-ball $B$ such that $\nu(B) \gtrsim \delta^{\eps'}$ for all $\eps'>\eps$.  Then
\[
Q+B \subseteq Q'
\]
for a cuboid $Q'$ with sidelengths comparable to the sidelengths of $Q$ and oriented in the same way as $Q$.  Then
\[
(\mu \ast \nu)(Q') \geq \mu(Q) \nu(B) \gtrsim \mathcal{L}^d(Q)^\gamma \delta^\eps \geq \mathcal{L}^d(Q)^{\gamma+\eps'}
\]
and so $\mu \ast \nu$ is $(\gamma+\eps')$-flat. If $\alpha \in (0,d)$ is the Frostman dimension of $\mu$ then the Frostman dimension of $(\mu \ast \nu) $ is at least $\alpha$. We may assume $\eps$ and $\eps'$ are small enough to ensure $\gamma+\eps'<1$ and then, after letting $\eps' \to \eps$, Corollary \ref{fourierspectrum} gives
\[
\fs (\mu \ast \nu)   \leq \frac{2(\gamma+\eps)(d-\alpha) }{1-(\gamma+\eps)} + \theta \cdot \frac{\alpha - d (\gamma+\eps)}{1-(\gamma+\eps)}
\]
for all $0 \leq \theta \leq 2(\gamma+\eps)$.  Since this holds for all $\gamma>0$ we conclude that
\[
 \kappa \theta  + \eps \leq \fs (\mu \ast \nu)   \leq \frac{2\eps(d-\alpha) }{1-\eps} + \theta \cdot \frac{\alpha - d \eps}{1-\eps}
\]
for all $0 \leq \theta \leq 2\eps$.  First, setting $\theta=0$ we get
\[
\eps \leq   \frac{2\eps(d-\alpha) }{1-\eps}  
\]
and since this must hold for all $\eps>0$, we get
\[
\alpha \leq d-1/2.
\]
 Second, setting $\theta=2\eps$ we get
\[
 2\eps \kappa   + \eps \leq  \frac{2\eps(d-\alpha) }{1-\eps} + 2\eps \frac{\alpha - d \eps}{1-\eps}
\]
which, upon dividing through by $2\eps$ and rearranging, gives
\[
\kappa \leq d-1/2,
\]
which contradicts the choice of $\kappa$ and establishes the desired bound for the Fourier spectrum.   Finally, the fact that $\mu$ is not $L^2$-flattening now follows from Theorem \ref{flatteningspectrum}.
\end{proof}

The upper bound of $(d-1/2)$ for the Sobolev dimension in Theorem \ref{flattening} is intriguing and a first question is whether it is optimal.  In fact, it is not and we prove something more general below.  Note that bounding the Sobolev dimension above by $(d-1)$ does \emph{not} improve on the upper bound for the Fourier spectrum in Theorem \ref{flattening} apart from for $\theta$ close to 1 and this is not what is relevant in the characterisation of $L^2$-flattening. We do not know if $(d-1/2) \, \theta$ is the optimal upper bound for the Fourier spectrum near zero, noting that $0$-flat examples with Fourier spectrum equal to $(d-1)\, \theta$ are easily constructed, see below.

\begin{thm}
If $\mu$ is a $\gamma$-flat measure on $\rd$  for $\gamma \in [0,1]$, then
\[
\frd \mu \leq d-1+\gamma
\]
and, provided $\gamma<1/2$,
\[
\sd \mu \leq d-1+2 \gamma.
\]
In particular, if $\mu$ is $\gamma$-flat for all $\gamma>0$, then
\[
\frd \mu \leq \sd \mu \leq d-1
\]
and this bound is sharp. 
\end{thm}

\begin{proof}
Let $Q_k$ be a sequence of cuboids witnessing the $\gamma$-flatness of $\mu$ and write
\[
0 < \delta_k^1 \leq \delta_k^2 \leq \cdots \leq \delta_k^d \leq 1
\]
for the sidelengths of $Q_k$.  Note that, by assumption,
\[
\mu(Q_k) \gtrsim (\delta_k^1   \cdots   \delta_k^d)^\gamma.
\]
Decompose $Q_k$ into 
\[
\approx \frac{\delta_k^1   \cdots   \delta_k^d}{(\delta_k^1)^d}
\]
many $\delta_k^1$-cubes.  By pigeonholing, there must exist such a $\delta_k^1$-cube of $\mu$ mass at least 
\[
\gtrsim \mu(Q_k) \frac{(\delta_k^1)^d}{\delta_k^1   \cdots   \delta_k^d} \gtrsim   (\delta_k^1   \cdots   \delta_k^d)^{\gamma-1}  (\delta_k^1)^d \geq (\delta_k^1)^{d-1+\gamma} 
\]
and this proves $\frd \mu \leq d-1+\gamma$.  For the Sobolev dimension bound, we have to work slightly harder.  Summing over all of the  $\delta_k^1$-cubes $Q$ formed from $Q_k$ and recognising that the second moment is minimised by the uniform distribution we get
\[
\sum_Q \mu(Q)^2 \gtrsim \frac{\delta_k^1   \cdots   \delta_k^d}{(\delta_k^1)^d} \left( (\delta_k^1   \cdots   \delta_k^d)^{\gamma-1}  (\delta_k^1)^d    \right)^2 = (\delta_k^1   \cdots   \delta_k^d)^{2\gamma-1}  (\delta_k^1)^d     \geq (\delta_k^1)^{d-1+2\gamma}. 
\]
This proves that $\sd \mu \leq d-1+2\gamma$, noting the use of $\gamma<1/2$ in the final inequality.  Here we are appealing to a well-known alternative formulation of the energy dimension in terms of summing over the measure of cubes and recalling that the Sobolev and energy dimensions coincide when they are strictly less than $d$, see \cite{peressolomyak}. 

The common upper bound of $(d-1)$ is then found by letting  $\gamma$ tend to zero. Finally, this common bound is sharp and this can be seen by taking the $(d-1)$-dimensional Lebesgue measure restricted to a cube and then embedded in $\rd$.  In particular, such a measure is $0$-flat but has Frostman (and Sobolev) dimension equal to $d-1$. In fact, the Fourier spectrum is given by $(d-1) \, \theta$ for all $\theta \in [0,1]$.
\end{proof}

\section{Background on fractal geometry and dimension theory}  \label{dimreview}

In this section we recall several further notions of fractal dimension, of both sets and measures, which we will use throughout the rest of the paper.  The goal now is to use familiar fractal geometric concepts to conclude $\gamma$-flatness and $(\gamma,\upsilon)$-flatness and then to apply the general results from the previous section in various settings of interest.

First, the \emph{upper box dimension} of a bounded set $X \subseteq \rd$ is defined by
\[
\ubd X = \limsup_{r \to 0} \frac{\log N_r(X)}{-\log r}
\]
where $N_r(X)$ is defined to be the smallest number of sets of diameter $r$ required to cover $X$.  Similarly, the \emph{lower box dimension} $\lbd X$ is defined by replacing the limsup with a liminf. If the upper and lower box dimensions coincide, then we refer to the common value as the \emph{box dimension} and denote it by $\bd X$.

The Assouad   spectrum, introduced in \cite{assouadspectrum}, interpolates between the upper box dimension  and the Assouad   dimension  in a meaningful way.  It  provides a parametrised family of dimensions by fixing the relationship between the two scales $r<R$ used to define Assouad  dimension.   Studying the dependence on the parameter within this family  thus yields finer and more nuanced information about the local structure of the set.  For example, one may understand which scales `witness' the behaviour described by the Assouad  dimension.  For $\theta \in (0,1)$, the \textit{Assouad spectrum} of $X$ is given by
\begin{align*}
\as X = \inf \Bigg\{ s \geq 0 :   \exists C>0 \ , \ \forall \ 0<r<1 \  ,  \ \forall x \in F \  ,  \ 
N_r(B(x,r^{\theta}) \cap X) \leq C \left(\frac{r^{\theta}}{r} \right)^{s} \Bigg\} .
\end{align*}
See \cite{jon:book} for more background and basic properties of the Assouad and lower spectra.  It was shown in \cite{assouadspectrum} that, for a bounded set $F \subseteq \mathbb{R}^d$, 
\begin{align}
\ubd X &\leq \as  X \leq    \frac{\ubd X}{1-\theta} . \label{basicbound} 
\end{align}
In particular, $\as X \to \ubd X$ as $\theta \to 0$.   The limit  $\lim_{\theta \to 1} \as X $ exists and equals the quasi-Assouad dimension, which is often but not always equal to the Assoaud dimension.  The quasi-Assouad dimension was introduced in \cite{quasiassouad}.

 The \textit{upper box dimension} of $\mu$  with bounded support is defined by
\begin{align*}
\ubd \mu = \inf \Big\{ s \geq 0 :  \exists C>0 \ , \ \forall \ 0<r< 1 \  , \ \forall x \in X \  ,  \  \mu(B(x,r)) \geq Cr^{s} \Big\} 
\end{align*}
and the \textit{lower box dimension} of $\mu$ is given by
\begin{align*}
\lbd \mu = \inf \Big\{ s \geq 0 : \exists  C>0 \  , \ \forall  r_0>0 \ , \  \exists \ 0<r<r_0 \ , \ \forall \ x\in X \  ,  \  \mu(B(x,r)) \geq Cr^{s} \Big\}.  
\end{align*}
If $\ubd \mu = \lbd \mu$, then we refer to the common value as the \textit{box dimension} of $\mu$, denoted by $\bd \mu$. These definitions of the box dimension of a measure were introduced    in \cite{antti}. It was shown that the bounds
\[
\ubd X \leq \ubd \mu
\]
and
\[
\lbd X \leq \lbd \mu
\]
always hold  and these inequalities are sharp in the sense that $X$ always supports a measure which attains equality in either of the above.  In this way, the box dimensions of a set are characterised by the box dimensions of measures they support.  

The upper box dimension of a measure is  dual to the Frostman dimension, $\frd \mu$ which we have already met. In general,
\[
\frd \mu \leq \sd \mu \leq \hd \mu \leq \hd X \leq \lbd X \leq \ubd X \leq \ubd \mu.
\]

\section{Quantitative flatness from geometry} 

\subsection{Basic estimates}

In this section we present some very simple applications of our `quantitative flatness' approach.  In this way we will see some well-known and straightforward facts recovered in this context. This serves as a warm up for later applications and as a sanity check against known constraints. 

One can see from Theorem \ref{knapp} that in order to obtain an obstruction to Fourier restriction, the aim is to find places where a lot of $\mu$ mass is found in a relatively small volume cuboid.  Of course, one can always take the cuboid to be a cube and the mass is then controlled by the Frostman dimension.  This basic application generalises  a known obstruction in the $p=2$ case.  More specifically, we recover a basic threshold for restriction due to  Hambrook--{\L}aba \cite{HL13} and also an estimate for the Fourier dimension due to  Mitsis \cite{Mit02}.

\begin{thm} \label{basic}
Let $\mu$ on $\rd$ be an arbitrary compactly supported measure with $\frd \mu<d$.  Then $\mu$ is $\gamma$-flat for all
\[
\gamma>\frac{\frd \mu}{d}.
\]
Further, 
\begin{enumerate}
\item If $\mu$   satisfies the  $E(p\to q)$ extension estimate \eqref{eq:extension}, then
\[
\frac{p}{q(p-1)} \leq \frac{\frd \mu}{d}.
\]
In particular, if $p=2$ then 
\[
q \geq \frac{2d}{ \frd \mu}.
\]
\item If $\mu$  satisfies the $I(p \to q)$ improving estimate \eqref{improvingest}, then
\[
\frac{1}{p}-\frac{1}{q} \leq \frac{\frd \mu}{d}.
\] 
\item The estimate
\[
\fd \mu \leq 2 \frd \mu
\]
holds.
\end{enumerate}
\end{thm}
\begin{proof}
Let $d>\alpha>\frd \mu$.  Then by definition we can find a sequence of cubes $Q_k$ with diameter tending to zero satisfying
\[
\mu(Q_k) \gtrsim \mathcal{L}^d(Q_k)^{\alpha/d}
\]
and therefore $\mu$ is $(\alpha/d)$-flat.  Applying Theorem \ref{knapp} we get 
\[
\frac{p}{q(p-1)} \leq \alpha/d
\]
and letting $\alpha \searrow \frd \mu$ gives the first result.  The second result follows similarly from Theorem \ref{thm:improving}.  For the final part, now suppose $\frd \mu <d$ and that $ \alpha<d$ and apply  Corollary \ref{fourier} to get 
\[
\fd \mu \leq \frac{2\frac{\alpha}{d}(d-\alpha) }{1-\frac{\alpha}{d}} = 2 \alpha
\] 
and letting $\alpha \searrow \frd \mu$ completes the proof.
\end{proof}

A trivial but important observation which is in some sense an extreme special case of quantitative flatness is that if a measure in $\rd$ is supported in an affine hyperplane (or merely gives positive mass to a hyperplane), then there are no non-trivial restriction estimates and the Fourier dimension is necessarily zero.  These observations follow easily from Theorem \ref{knapp} and Corollary \ref{fourier} but we also obtain the following more general observation which says that being sufficiently concentrated near an affine hyperplane gives the same conclusion.

\begin{cor}
Suppose $\mu$ is a compactly supported measure in $\rd$ and $V \subseteq \rd$ is an affine hyperplane such that
\begin{equation} \label{decayestimatehp}
\mu( V_\delta) \gtrsim \delta^s
\end{equation}
for all sufficiently small $\delta>0$ and some $s \in ( 0,1)$.     Then   $\mu$ is $s$-flat and, therefore, 
\begin{enumerate}
\item  If $\mu$ satisfies the $E(p \to q )$  extension estimate \eqref{eq:extension}, then
\[
\frac{p}{q(p-1)} \leq s.
\]
In particular, if $p=2$, then  $q \geq 2/s$. 
\item If $\mu$  satisfies the $I(p \to q)$ improving estimate \eqref{improvingest}, then
\[
\frac{1}{p}-\frac{1}{q} \leq s.
\] 
\item The estimate, 
\[
\fd \mu \leq \frac{2s (d-\alpha)}{1-s}
\]
holds, where $\alpha$ is the Frostman dimension of $\mu$.
\end{enumerate}
 In particular, if $s>0$ can be taken arbitrarily small, for example, if the right hand side of \eqref{decayestimatehp} is logarithmic, then there  are no non-trivial restriction estimates or $L^p$-improving estimates  and $\fd \mu = 0$.  The case when $\mu$ gives positive measure to a hyperplane corresponds to $s=0$.
\end{cor}

\begin{proof}
Let $Q \subseteq V_\delta$ be a cuboid which contains the support of $\mu$ intersected with $V_\delta$ and has one side of length $\approx \delta$ and the remaining sides of length $\approx 1$.  Then
\[
\mu(Q) \geq \mu(V_\delta) \gtrsim \delta^s
\]
and so $\mu$ is $s$-flat  and the results follow from  Theorem \ref{knapp}, Theorem \ref{thm:improving} and Corollary \ref{fourier}.
\end{proof}

\subsection{$C^2$ surfaces:~using curvature} \label{sec:surfaces}

Many of the key examples in Fourier restriction theory are smooth manifolds, such as the sphere, cone, or hyperboloid.  It is also of interest to study smooth manifolds in general, see, for example, \cite{bak,leevargas, mattilaFourier, muller}.  Here we show how quantitative flatness can be used in this setting by leveraging (lack of) curvature in the manifold.

We consider general  $C^2$ manifolds $X \subseteq \rd$ of dimension $s \in \{1, \dots, d-1\}$.  Such a manifold comes with a natural surface measure, which is the analogue of $s$-dimensional Lebesgue measure on the surface and these are our main measures of interest.    A point $x \in X$  is said to be of   \emph{rank} $k$ if  there are $k$ non-zero principal curvatures and $s-k$ principal curvatures equal to zero.  Here $s-k$ is the \emph{nullity}.  Moreover, we say a point  $x \in X$  is   of   \emph{ambient rank} $k$ if 
\[
r^*(x) := \limsup_{y \to x} r(y) = k
\]
where $r(y)$ denotes the rank of $y$.  We say $n^*(x) = s-r^*(x)$ is the \emph{ambient nullity} of $x$.  Then $r^*(x) \geq r(x)$ for all $x$ but this inequality can be strict.  For example, 
\[
X = \{(x,x^2): x \in [-1,1]\}
\]
is a $C^2$ manifold of dimension 1 in the plane.  The rank of every point is maximal (equal to 1) apart from at $(0,0)$ where it is 0.  However, the ambient rank of every point is 1.  Of course, if $X$ is of constant rank, that is, $r(x)$ is the same for all $x \in X$, then ambient rank and rank coincide at all points.

There is considerable interest in both the Fourier restriction and Fourier decay problem for these surface measures.  The Fourier restriction theory is fairly classical but the Fourier decay problem has seen a lot of attention recently.  For example, it was proved in \cite{zhu2} that any smooth constant rank $k$ hypersurface (that is, of co-dimension $d-s=1$) has Fourier dimension 
\[
\fd X = k.
\]
It is an interesting problem to relax the assumptions, especially the constant rank and co-dimensionality assumptions.  We make progress in both of these  directions.
\begin{thm} \label{smoothsurface1}
Let $\mu$ be the surface measure on a bounded $C^2$ manifold $X \subseteq \rd$ of dimension $s \in [1,d-1]$ and let $x \in X$.  Then $\mu$ is $(\gamma, \upsilon)$-flat for
\[
\gamma = \frac{r^*(x)}{r^*(x)+2(d-s)}
\]
and
\[
\upsilon = r^*(x)/2+d-s
\]
where $r^*(x)$ is the ambient rank of $x$.
Therefore:
\begin{enumerate}
\item If  the $E(p\to q)$ extension estimate 
\eqref{eq:extension} holds for $\mu$, then
\[
\frac{p}{q(p-1)} \leq \min_{x \in X}\frac{r^*(x)}{r^*(x)+2(d-s)},
\]
and, for $p=2$, 
\[
q \geq \max_{x \in X}\frac{2r^*(x)+4(d-s)}{r^*(x)} .
\]
\item If $\mu$  satisfies the $I(p \to q)$ improving estimate \eqref{improvingest}, then
\[
\frac{1}{p}-\frac{1}{q} \leq \min_{x \in X}\frac{r^*(x)}{r^*(x)+2(d-s)}.
\] 
\item The bound
\[
\fd \mu \leq \min_{x \in X} r^*(x)
\]
  holds. In particular, if there exists a point with non-trivial ambient nullity, then $\mu$ is not a Salem measure. Further, for all $x \in X$
  \[
  \fs \mu \leq  r^*(x)+\big(s-r^*(x)/2 \big)\theta 
  \]
  for all $0 \leq \theta \leq 2r^*(x)/\left(r^*(x)+2(d-s)\right)$. This significantly beats the general estimate for the Fourier spectrum coming from the Fourier dimension and \eqref{bound} since   $s-r^*(x)/2 <d$.
\end{enumerate}
\end{thm}

\begin{proof}
First note that since the ambient rank is an integer, the minimum 
\[
\min_{x \in X} r^*(x)
\]
is realised and so we may fix a minimising $x$.  Then, by \cite[Lemma 3.1]{hartman}, there is a ball $B_0 \subseteq \mathbb{R}^{s-r^*(x)}$ which embeds isometrically into $X$ with the centre of the ball mapping to $x$. Write $B \subseteq X$ for the embedding of this ball.  Let $V \subseteq \rd$ be the $s$-dimensional tangent plane of $X$ at $x$.  This tangent plane  exists because $X$ is $C^2$.  In particular,  $x \in B \subseteq V$. For $\delta>0$ small, first let
\[
Q' = B_{\delta^{1/2}} \cap V
\]
be the $\delta^{1/2}$ neighbourhood of $B$ inside $V$ and then let 
\[
Q'' = Q'_\delta
\]
be the $\delta$ neighbourhood of $Q'$ inside $\rd$.  Then $Q''$ contains a cuboid $Q$ centred at $x$ with $s-r^*(x)$ many sides of length $\approx 1$ oriented in the `flat directions' inside $B$, $r^*(x)$ many sides of length $\approx \delta^{1/2}$ in the non-flat directions inside $V$ but orthogonal to $B$, and $d-s$ many sides of length $\approx \delta$ in the directions in the normal space of $X$ at $x$, that is, the orthogonal complement of $V$ at $x$.  Then
\[
\mathcal{L}^d(Q) \approx \delta^{r^*(x)/2+d-s}.
\]
Moreover, since $X$ is $C^2$, and $Q$ extends a distance $\delta^{1/2}$ in the non-flat directions inside $V$ it captures a $\approx \delta^{1/2}$ portion of $X$ in those directions too.  That is, $X$ cannot curve away from $V$ fast enough to escape $Q$.  Then, since $\mu$ is the surface measure on $X$,
\[
\mu(Q) \approx \delta^{r^*(x)/2}.
\]
Therefore, $\mu$ is $(\gamma, \upsilon)$-flat for 
\[
\gamma = \frac{r^*(x)/2}{r^*(x)/2+d-s}
\]
and
\[
\upsilon = r^*(x)/2+d-s
\]
as required.  The estimates in (1) and (2)  then follow immediately from Theorem \ref{knapp} and Theorem \ref{thm:improving}, respectively.  Turning our attention to the Fourier dimension and (3), noting that the Frostman dimension of the surface measure is the dimension of the ambient manifold, that is, $\frd \mu = s$, applying Corollary \ref{fourier} we get
\[
\fd\mu  \leq \frac{2 \gamma(d-s)}{1-\gamma} =  r^*(x)
\]
as required.  Finally, the bounds for the Fourier spectrum follow similarly from Corollary \ref{fourierspectrum}.
\end{proof}
The previous result established that the Fourier dimension of the surface measure on a $C^2$ manifold of arbitrary dimension and variable rank is bounded above by the smallest ambient rank of a point.  This estimate is as good as one can hope for in this level of generality.  In particular, the ambient rank cannot be replaced by the rank, which is what one might have naively conjectured. Indeed,  let us reconsider the simple example 
\[
X = \{(x,x^2): x \in [-1,1]\}.
\]
Here the minimal rank is zero but the minimal ambient rank is 1.  However, it follows easily from a stationary phase argument that the Fourier dimension of the surface measure is, in fact, 1.

Next we consider arbitrary measures on $X$ and the Fourier dimension of the manifold itself.  In the constant rank $r$ case, it is natural to conjecture that 
\begin{equation} \label{surfaceconjecture}
\fd X \leq r.
\end{equation}
In general nothing more can be said, because if $s \leq d-2$, then it is possible for $X$ to be contained in a hyperplane, even when the rank is maximal, in which case the Fourier dimension is zero.  However, if $s=d-1$, that is, we are in the setting of co-dimension 1 hypersurfaces, then  it follows by a straightforward stationary phase argument that the Fourier dimension of a small piece of the surface measure is at least $r$, which gives the corresponding lower bound.  The main interest is therefore in the upper bound \eqref{surfaceconjecture}, but bounding the Fourier dimension of $X$ above requires bounding the Fourier dimension uniformly above for \emph{all} measures on $X$ and this is  difficult.  The above conjecture \eqref{surfaceconjecture} has recently been verified in the co-dimension 1 case, that is, $s=d-1$ \cite{zhu2}.  In the following subsection we provide an alternative approach to this problem and show that if the dimension of the manifold is small enough, then \eqref{surfaceconjecture} is not optimal and can be improved in certain cases. In the following, we verify \eqref{surfaceconjecture} in the constant curvature arbitrary co-dimension case   under the assumption $r/2+s<d$.  Since there is nothing to prove when $r=s$ because then $\fd X \leq \hd X = s$, we may assume $r \leq s-1$.  Therefore,   we resolve the (constant rank) conjecture in all cases when 
\[
s<\frac{2d+1}{3}.
\]
In the variable rank case, we show the \emph{largest} ambient rank is an upper bound. At first sight this might seem sub-optimal because in Theorem \ref{smoothsurface1} we proved that the \emph{smallest} ambient rank is an upper bound for the Fourier dimension of the surface measure.  However, when one is considering the manifold itself,  in the variable rank case one may choose measures which avoid the locations with lowest rank and so instead the largest rank becomes the relevant quantity.  For example, suppose $X$ contains a  genuinely flat plane of co-dimension 1 and a curved part (where we glue the pieces together in a $C^2$ way).  Then the smallest ambient rank is zero, but the Fourier dimension can be strictly positive by witnessing this via measures supported only on the curved part.

\begin{thm} \label{smoothsurface2}
Let $X \subseteq \rd$ be a bounded $C^2$ manifold  of dimension $s \in [1,d-1]$.  Let
\[
r^*=\max_{x \in X} r^*(x)
\]
be the largest ambient rank over $X$. Then
\[
 \fd X \leq \frac{r^*d}{d+r^*/2-s}.
\]
 Moreover, if we assume that $r^*/2+s< d$, then 
\[
 \fd X \leq r^*
\]
thus  resolving \eqref{surfaceconjecture} (and extending it to variable rank) in this case.
\end{thm}

\begin{proof}
For every point $x \in X$, the ambient nullity is at least $s-r^*$ and  therefore, by \cite[Lemma 3.1]{hartman}, there is a $C^2$ foliation of $X$ by leaves given by isometric copies of pieces of $\mathbb{R}^{s-r^*}$. Now, let $\nu$ be an arbitrary non-zero finite Borel measure on $X$. We may assume that $\nu$ is supported on a part of $X$ foliated by leaves which  contain an $(s-r^*)$-dimensional ball of diameter uniformly bounded below, that is $\gtrsim 1$, over the whole foliation.  Indeed, we may write  $X$ as a countable union of pieces foliated by leaves containing balls of diameter at least $1/m$ and then restrict to one of these pieces with positive $\nu$-mass. By taking a maximal $\delta^{1/2}$ separated set in the $r^*$-dimensional parameter space underlying the foliation, we can find a subset of the foliation of size at most $\lesssim \delta^{-r^*/2}$ such that the $\approx \delta^{1/2}$ neighbourhoods (inside $X$) of the leaves in the subset form a cover of $X$.  Therefore, by pigeonholing, there must exist an element $L$  from the above constructed cover of $X$  which satisfies
\[
\nu(L) \gtrsim \delta^{r^*/2}.
\]
The leaf associated with $L$ contains an  $(s-r^*)$-dimensional flat ball and the $s$-dimensional tangent space of $X$ is constant on this ball. Now, let $Q$ be a cuboid containing   the $\delta$-neighbourhood of $L$ (in $\rd$)  with $s-r^*$ many sides of length $\approx 1$  oriented in the `flat directions' of the leaf associated with $L$, $r^*$ many sides of length $\approx \delta^{1/2}$ in the (potentially) non-flat directions inside the tangent space of $X$ associated with the leaf  but orthogonal to the leaf, and $d-s$ many sides of length $\approx \delta$ in the directions in the normal space of $X$. Here we are again using that $X$ is $C^2$ and so cannot curve away from its tangent planes any faster than quadratic. Then, once again
\[
\mathcal{L}^d(Q) \approx \delta^{r^*/2+d-s}
\]
and
\[
\nu(Q) \gtrsim \nu(L) \gtrsim \delta^{r^*/2},
\]
and so $\nu$ is $(\gamma, \upsilon)$-flat for 
\[
\gamma = \frac{r^*/2}{r^*/2+d-s}
\]
and
\[
\upsilon = r^*/2+d-s.
\]
This time applying Corollary \ref{fourier} is sub-optimal because we do not know the Frostman dimension of $\nu$ and we are left to apply Corollary \ref{fouriernofrost} which   yields the first general but not optimal bound on the Fourier dimension of $\nu$ and thus $X$.  To obtain the optimal bound,  we seek to apply Corollary \ref{fourier2}. For this we need
\[
\gamma \upsilon = r^*/2< d-s
\]
because the shortest $d-s$ sides of $Q$ are of size $\approx \delta$.  Under this assumption, 
\[
\fd \nu \leq 2\gamma \upsilon = r^*,
\]
as required. Since $\nu$ was arbitrary we get the desired upper bound for $\fd X$.
\end{proof}

\subsection{Smooth surfaces:~using dimensionality} \label{sec:surfaces2}

By following an approach of Gressman \cite{gressman} we next provide an alternative strategy for bounding the Fourier dimension of smooth submanifolds where we leverage low dimensionality, rather than low rank as in the previous subsection.  We arrive at some surprising conclusions, including that smooth curves in $\rd$ have Fourier dimension at most $4/(d+1)$.

We again consider $s$-dimensional submanifolds in $\rd$. Gressman \cite{gressman} introduced the quantity $H=H(s,d)$ which he called the \emph{homogeneous dimension} and is  equal to the smallest positive integer which can be written as  the sum of the degrees of some collection of $d$ distinct, nonconstant monomials in $s$ variables. Interpreting the main result of \cite{gressman} in our language, we get the following.

\begin{prop}(Gressman \cite[Theorem 1 (1)]{gressman})
If $\mu$ is a non-zero finite measure on a smooth immersed   $s$-dimensional    submanifold of $\rd$, then $\mu$ is $\gamma$-flat for all
\[
\gamma>\frac{s}{H(s,d)}.
\]
\end{prop}

Combining the previous result with Corollary \ref{fourier}, Corollary \ref{fourierspectrum} and Mitsis' Frostman dimension bound (see Corollary \ref{fouriernofrost}) we get an alternative way to bound the Fourier dimension of a submanifold which is especially good when $s$ is small. Note that there are no assumptions regarding rank or nullity in what follows.

\begin{cor} \label{surfacedim1}
Let  $X \subseteq \rd$  be a smooth    $s$-dimensional submanifold.  If $\mu$ is  an arbitrary non-zero finite compactly supported measure on $X$, then
\[
\fd \mu \, \leq \, \min\left\{ \frac{2s(d-\alpha)}{H(s,d)-s}  , \  2\alpha \right\} \,  \leq \,  \frac{2sd}{H(s,d)}
\]
where $\alpha$ is the Frostman dimension of $\mu$ and $H(s,d)$ is the homogeneous dimension.  Further,  if $\mu$ is (a compact piece of) the surface measure, then $\alpha=s$ and
\[
\fd \mu \leq \frac{2s(d-s)}{H(s,d)-s}+ \theta \cdot \frac{s(H(s,d)-d)}{H(s,d)-d}
\]
for all $\theta \in [0,2s/H(s,d)]$. Finally, we can bound the dimension of the manifold itself by
\[
\fd X \leq \frac{2sd}{H(s,d)}.
\]
\end{cor}

The bound for the Fourier dimension of the manifold $X$ itself is not useful when $s$ is large.  Indeed,   $H(d-1,d) = d+1$ since there are $d-1$ distinct  linear monomials and the next smallest monomial is necessarily of degree 2,  and so the bound exceeds $s = \hd X$ when $s=d-1$.  However, when $s$ is small, something interesting happens.  We first consider    curves where computing $H(s,d)$ is straightforward.

\begin{cor} \label{surfacedim2}
Suppose $X$ is a smooth curve (a 1-dimensional manifold) in $\rd$.  Then
\[
\fd X \leq \frac{4}{d+1}
\]
and so, when $d\geq 4$, $X$ is not Salem.  Moreover, the surface (arclength) measure $\mu$  on $X$ satisfies
\[
\fd \mu \leq \frac{4}{d+2}.
\]
\end{cor}

\begin{proof}
When $s=1$, 
\[
H(1,d) = \frac{d(d+1)}{2}
\]
and the results follow from Corollary \ref{surfacedim1}.  This is because we only have 1 variable and so there is only one distinct monomial for each degree and so $H(1,d) = 1+ 2+\cdots +d$.  
\end{proof}

Next we consider higher dimensional submanifolds where estimating $H(s,d)$ becomes a little more difficult. Again we see that when $d$ is sufficiently large an $s$-dimensional manifold cannot be Salem in $\rd$ and our bound for the Fourier dimension goes to zero as $d \to \infty$.

\begin{cor} \label{surfacedim3}
Let $X$ be a smooth $s$-dimensional submanifold in $\rd$ for $s \geq 2$.  If
\[
d \geq    \frac{s(s+1)}{2} + 2s+1
\]
then
\[
\fd X  \leq  \frac{2sd}{2d+1} < s
\]
 and $X$ is not a Salem set.  Moreover,  under the same assumption on $d$, the surface  measure $\mu$  on $X$ satisfies
\[
\fd \mu \leq \frac{2s(d-s)}{2d+1-s}.
\]
Further, if we fix $s$ and consider large $d$, then
\[
\fd X  \leq  C_s d^{-1/s}
\]
where $C_s$ is a constant depending only on $s$. 
\end{cor}

\begin{proof}
For $s \geq 2$,  there are $s$ distinct linear monomials and 
\[
\binom{s+1}{2}
\]
many distinct monomials of degree 2.  Therefore, if 
\[
d =  s+ \binom{s+1}{2} +(s+1) +\kappa =  \frac{s(s+1)}{2} + 2s+1+\kappa,
\]
for some $\kappa \geq 0$, then
\[
H(s,d)  \geq s+ 2 \binom{s+1}{2} +3(s+1+\kappa) = 2d+1+\kappa \geq  2d+1
\]
and the explicit bounds  follow from Corollary \ref{surfacedim1}.  With $s$ fixed and $d$ large, choose  $r$ to be the unique integer such that
\[
\sum_{k=1}^{r-1}  \binom{k+s-1}{k} = \binom{r+s-1}{s}-1     \leq d <\binom{r+s}{s}-1 
\]
noting that $d \lesssim r^s$.  By a `stars and bars' count, there are 
\[
 \binom{k+s-1}{k}
\]
distinct monomials of degree $k$ in $s$ variables and so, by the definition of $r$, we must use all of the monomials of degree $<r$ when defining $H(s,d)$.  Therefore,
\[
H(s,d) \geq \sum_{k=1}^{r-1} k \binom{k+s-1}{k} = \frac{r(r-1)}{s+1}\binom{r+s-1}{r}  \gtrsim_s r^{s+1}\gtrsim d^{1+1/s}
\]
for $d$ (and $r$) large compared to $s$.  The bound 
\[
\fd X \leq \frac{2sd}{H(s,d)} \leq C_s d^{-1/s}
\]
then follows from Corollary \ref{surfacedim1}.
\end{proof}

\subsection{Fractal surfaces}

In this section we remain in the setting of surfaces, but we drop all smoothness assumptions.  We consider general $(d-k)$-dimensional topological  manifolds  which we assume are the graphs of H\"older  functions $f: [0,1]^{d-k} \to \mathbb{R}^k$ where the graph is defined by
\[
G_f= \{(x,f(x)) : x \in [0,1]^{d-k} \} \subseteq  \rd.
\]
Then we replace the surface measure by the lift of Lebesgue measure on $[0,1]^{d-k}$ to the graph.   Without derivatives, the standard Knapp style examples no longer apply and instead we prove quantitative flatness using other geometric properties of the surface including H\"older continuity and fractal dimensions.   First, we obtain a basic estimate coming directly from the H\"older exponent. We get useful information in the context of restriction and $L^p$-improving properties, but not in terms of Fourier dimension. The special case $k=1$ of this estimate was already  observed in  \cite[Theorem 2.6]{sheet}.

\begin{thm} \label{surface1}
Let $1 \leq k < d$ be integers and suppose $\mu$ is the lift of Lebesgue measure on  $[0,1]^{d-k}$ to the graph of an $\alpha$-H\"older function $f: [0,1]^{d-k} \to \mathbb{R}^k$.  Then $\mu$ is $(\gamma, \upsilon)$-flat for 
\[
\gamma = \frac{d-k}{d-k+k\alpha}
\]
and 
\[
\upsilon = d-k+k\alpha.
\]
Therefore, if the $E(p \to q)$ extension estimate \eqref{eq:extension} holds, then
\[
\frac{p}{q(p-1)} \leq \frac{d-k}{d-k+k\alpha}.
\]
In particular, if $p=2$ then 
\[
q \geq 2+ \frac{ 2k\alpha}{d-k}.
\]
Further,  if $\mu$  satisfies the $I(p \to q)$ improving estimate \eqref{improvingest}, then
\[
\frac{1}{p}-\frac{1}{q} \leq  \frac{d-k}{d-k+k\alpha}. 
\] 
\end{thm}
\begin{proof}
Let $C \geq 1$ be the H\"older constant of  $f$ and let $Q$ be any axes oriented cuboid with side lengths 
\[
\underbrace{\delta \times \cdots \times \delta }_{\text{$(d-k)$ times}} \times \underbrace{2C\delta^\alpha \times \cdots \times 2C\delta^\alpha}_{\text{$k$ times}}
\]
 centred at a point on the graph for some $\delta>0$.  Then 
\[
\mathcal{L}^{d}(Q) = (2C)^{k} \delta^{d-k+k\alpha}
\]
and 
\[
\mu(Q)  = \delta^{d-k}
\]
since the graph restricted to the projection of $Q$ onto $[0,1]^{d-k}$ is entirely contained in $Q$. This follows from the H\"older condition. Then
\[
\mu(Q) \approx \mathcal{L}^{d}(Q)^{\frac{d-k}{d-k+k\alpha}}
\]
and the claimed $(\gamma, \upsilon)$-flatness follows.  The remaining results  then follow from Theorems \ref{knapp} and \ref{thm:improving}.
\end{proof}

Our next goal is to beat the previous estimate by replacing the `vertical cuboids' from the previous result with `horizontal cuboids' which take advantage of fractal features of the surface.  We can obtain such fractal information by considering the Assouad spectrum of `level sets' $f^{-1}(z)$, for $z \in \mathbb{R}^k$ in the range.  For complicated functions $f^{-1}(z) \subseteq [0,1]^{d-k}$ is often a fractal set and may have non-trivial dimensional information.  In particular, large level sets coupled with the H\"older condition guarantee lots of mass is trapped in a relatively flat  cuboid. We assume $\theta \in [0,\alpha]$ to force the cuboids we use in the proof to be horizontally flat and so genuinely distinct from the vertical cuboids considered in Theorem \ref{surface1}. Otherwise Theorem \ref{surface1} will be better, because the H\"older condition traps the whole surface in the cuboid without appealing to dimension theory of level sets.

\begin{thm} \label{surface2}
Suppose $\mu$ is Lebesgue measure lifted to the graph of an $\alpha$-H\"older function $f: [0,1]^{d-k} \to \mathbb{R}^k$.  Then, for all $z \in \mathbb{R}^k$ and $\theta \in [0,\alpha]$, $\mu$ is $(\gamma, \upsilon)$-flat for all
\[
\gamma > \frac{d-k-\as f^{-1}(z) (1-\theta)}{(d-k)\theta+k\alpha}
\]
and
\[
\upsilon = \frac{(d-k)\theta}{\alpha}+k.
\]
Therefore:
\begin{enumerate}
\item If $\mu$ satisfies the $E(p \to q )$  extension estimate \eqref{eq:extension}, then
\[
\frac{p}{q(p-1)} \leq \frac{d-k-\as f^{-1}(z) (1-\theta)}{(d-k)\theta+k\alpha}
\]
for all $\theta \in [0,\alpha]$. In particular, if $p=2$ then 
\[
q \geq \frac{2((d-k)\theta+k\alpha)}{d-k-\as f^{-1}(z) (1-\theta)}.
\]
Setting $\theta=0$ we get that for all $z \in \mathbb{R}$
\[
\frac{p}{q(p-1)} \leq \frac{d-k-\ubd f^{-1}(z) }{k\alpha}
\]
and, for $p=2$, 
\[
q \geq \frac{2k\alpha}{d-k-\ubd f^{-1}(z)}.
\]
\item If $\mu$  satisfies the $I(p \to q)$ improving estimate \eqref{improvingest}, then
\[
\frac{1}{p}-\frac{1}{q} \leq  \frac{d-k-\as f^{-1}(z) (1-\theta)}{(d-k)\theta+k\alpha}. 
\] 
\item The estimate
\[
\fd \mu \leq  2k \, \frac{d-k-\as f^{-1}(z) (1-\theta)}{(\as f^{-1}(z) -d+k)(1-\theta)+k\alpha}
\]
holds for all $\theta \in [0,\alpha]$. In particular, setting $\theta=0$ we get
\[
\fd \mu \leq 2k \, \frac{d-k-\ubd f^{-1}(z) }{\ubd f^{-1}(z) -d+k+k\alpha}.
\]
Further, provided 
\[
\as f^{-1}(z) > \frac{d-k-k\alpha}{1-\theta},
\]
then the estimate
\[
\fd \mu \leq  2 \, \frac{d-k-\as f^{-1}(z) (1-\theta)}{\alpha}
\]
holds for all $\theta \in [0,\alpha]$. In particular, setting $\theta=0$ we get
\[
\fd \mu \leq  2 \, \frac{d-k-\ubd f^{-1}(z)}{\alpha}
\]
provided $\ubd f^{-1}(z) >d-k-k\alpha$.
\end{enumerate}
In the extreme case when there exists $z$ such that $\ubd f^{-1}(z)=d-k$ (or even approaches $d-k$) then there are no non-trivial Fourier restriction or $L^p$-improving  estimates at all, and $\mu$ is not $L^2$-flattening.  Moreover, $\fd \mu = 0$, the Frostman and Sobolev dimensions of $\mu$ are at most $d-1/2$ and 
\[
\fs \mu \leq  (d-1/2) \, \theta 
\]
for all $0 \leq \theta \leq 1$.
\end{thm}
\begin{proof}
Let $z \in \mathbb{R}^k$ and let
\[
s<\as f^{-1}(z) \leq d-k.
\]
By definition there exists a sequence of scales $r \to 0$ and associated  cubes $Q' \subseteq [0,1]^{d-k}$ with sidelengths $r^\theta$ such that we may find an $r$-separated subset $X$ of $f^{-1}(z) \cap Q'$ of size at least $\gtrsim r^{s(\theta-1)}$.  Let $Q$ be the cuboid given by
\[
Q = Q' \times B(z, C r^\alpha).
\]
Then $Q$ has $(d-k)$ many `long' sides of length $\approx r^\theta$ and $k$ many `short' sides of length $\approx r^\alpha$ and so
\[
\mathcal{L}^{d}(Q) \approx  r^{(d-k)\theta+k\alpha} = \big(  r^\alpha \big)^{\frac{(d-k)\theta}{\alpha}+k}.
\]
Moreover, the graph of $f$ restricted to an $r$-ball centred at each point of $X$ lies completely inside $Q$ by the H\"older property and, moreover, these restricted graphs are pairwise disjoint and are of $\mu$ measure $\approx r^{d-k}$. Therefore
\[
\mu(Q)  \gtrsim r^{s(\theta-1)} r^{d-k} \approx  \mathcal{L}^{d}(Q)^{\frac{s(\theta-1)+d-k}{(d-k)\theta+k\alpha}}.
\]
The claimed $(\gamma,\upsilon)$-flatness follows.  The results in (1) and (2) then follow from Theorems \ref{knapp} and \ref{thm:improving}, respectively. Turning our attention to the Fourier dimension and (3), first observe that the Frostman dimension of $\mu$ is at least $d-k$ due to the fact that
\[
\mu(B(x,r)) \leq \mathcal{L}^{d-k}(\pi(B(x,r))) \approx r^{d-k}
\] 
for all $x$ and $r$ and where $\pi$ denotes orthogonal projection onto the domain. Then Corollary \ref{fourier} gives
\[
\fd \mu \leq \frac{2 \gamma}{1-\gamma} (d-(d-k)) = \frac{2 k\gamma}{1-\gamma}   
\]
and plugging in $\gamma$ and simplifying gives the desired bound.  The final claim in (3) follows from Theorem \ref{fourier3}.  To be in a position to apply this estimate we need $\gamma \upsilon<k$ since the $k$ shortest sides of $Q$ above are $\approx r^\alpha$.  This can be achieved under the additional assumption.

Finally, the results in the extreme case follow from the previous estimates together with Theorem \ref{flattening}.
\end{proof}

Theorem \ref{surface2} already beats Theorem \ref{surface1} in the $\theta=0$ case whenever there is a level set with upper box dimension
\[
> \frac{(d-k)^2}{d-k+k\alpha}
\] 
and it can beat it with even less restrictive bounds on the Assouad spectrum. The Fourier dimension bound is also interesting.  In this setting 
\[
\fd \mu \leq \fd \mathcal{L}^{d-k}\vert_{[0,1]^{d-k}} = 2
\]
always holds since projection cannot decrease Fourier dimension.  Therefore we are interested in conditions which beat this estimate.  At $\theta=0$, this will be achieved   whenever 
\[
\ubd f^{-1}(z) >d-k-\frac{k\alpha}{k+1}
\]
for some $z$, which \emph{can} be satisfied since the level sets live in $(d-k)$-dimensional space. More generally, one may derive a condition in terms of the Assouad spectrum and which $\theta$ is optimal is not pre-determined. That said, the bound at $\theta=\alpha$ is never better than 2 and so one looks for intermediate $\theta$s.

Concerning the second (conditional) Fourier dimension bound, note that it is always better than the first when it applies.  Moreover, when $k \geq d/(1+\alpha)$ it always applies.  

\begin{example}
Let $\mu$ be Lebesgue measure lifted onto the graph of the function $f:(0,1] \to \mathbb{R}$ defined by
\[
f(x) = 2^{-\alpha/x} \sin(2 \pi 2^{1/x}).
\]
for some $\alpha \in (0,1)$.  Then there are no non-trivial Fourier restriction or $L^p$-improving estimates for $\mu$ and $\mu$ is not $L^2$-flattening.  Moreover, $\fd \mu = 0$, the Frostman and Sobolev dimensions of $\mu$ are at most $3/2$ and 
\[
\fs \mu \leq  3\theta/2 
\]
for all $0 \leq \theta \leq 1$.

This follows from Theorem \ref{surface2} since $f$ is a  H\"older function and there is a level set with box dimension 1, namely the zero set of the function.
\end{example}

\begin{example}
Let $\mu$ be Lebesgue measure lifted onto the graph of the function $f:(0,1] \to \mathbb{R}$ defined by
\[
f(x) = x^{1/\alpha} \sin(2 \pi x^{-1/\beta}).
\]
for some $\alpha \in (0,1)$ and $\beta>0$ such that $\frac{\beta}{\beta+1} \leq \alpha$.  Then $f$ is $\alpha'$-H\"older for all $\alpha'<\frac{\beta}{\alpha(\beta+1)}$ and there is a level set, namely, the zero set, which has Assouad spectrum
\[
\as f^{-1}(0) = \min\left\{\frac{1}{(1-\theta)(1+\beta)}, \, 1 \right\}
\]
see \cite{assouadspectrum}. Therefore, Theorem \ref{surface2} yields that:
\begin{enumerate}
\item If $\mu$ satisfies the $E(p \to q )$  extension estimate \eqref{eq:extension}, then
\[
\frac{p}{q(p-1)} \leq \frac{1-\min\left\{\frac{1}{1+\beta}, \, 1-\theta \right\}}{\theta+\frac{\beta}{\alpha(\beta+1)}}
\]
and this is optimised at the phase transition $\theta=\beta/(1+\beta)$ to give 
\[
\frac{p}{q(p-1)} \leq \frac{\alpha}{\alpha+1}
\]
and, in the case $p=2$,
\[
q \geq 2+\frac{2}{\alpha}.
\]
\item Similarly, if $\mu$  satisfies the $I(p \to q)$ improving estimate \eqref{improvingest}, then
\[
\frac{1}{p}-\frac{1}{q} \leq  \frac{\alpha}{\alpha+1}. 
\] 
\item The estimate
\[
\fd \mu \leq  \frac{2-2\min\left\{\frac{1}{(1-\theta)(1+\beta)}, \, 1 \right\} (1-\theta)}{(\min\left\{\frac{1}{(1-\theta)(1+\beta)}, \, 1 \right\}-1)(1-\theta)+\frac{\beta}{\alpha(\beta+1)}}
\]
holds. Again this is optimised at the phase transition $\theta=\beta/(1+\beta)$ to give 
\[
\fd \mu \leq  2\alpha.
\]
\end{enumerate}
\end{example}

Next we apply the converse of Theorem \ref{surface2} to  get uniform information about the box dimension of the level sets of the graphs of stochastic processes. These estimates are   non-optimal, but at least give non-trivial estimates via a rather different approach.

\begin{cor}
Almost surely \emph{all} of the  level sets of the graph of fractional Brownian motion with Hurst parameter $H \in (0,1)$  on $[0,1]$  have upper box dimension bounded above by $1-H/3$.  
\end{cor}

\begin{proof}
It is well-known that factional Brownian motion is $\alpha$-H\"older for all $\alpha<H$ and we know from \cite{frasersahlsten, fractional1,fractional2} that the Fourier dimension of the graph  of   fractional Brownian is 1 (independent of $H$).  Moreover, the measure witnessing this fact is the pushforward of Lebesgue measure onto the graph.  Therefore, there cannot exist a level set with upper box  dimension exceeding $1-H/3$ because otherwise the ($k=1, d=2$) bound from Theorem \ref{surface2} would give an upper bound strictly less than 1 for the Fourier dimension of the pushforward of Lebesgue measure onto the graph.
\end{proof}

It is well-known that almost surely for Lebesgue positively many $z$ the level sets of classical Brownian motion  have Hausdorff dimension $1-H$.  The fact that we obtain $1-H/3$  as an upper bound for \emph{all} level sets and with Hausdorff dimension replaced by the \emph{a priori}  larger upper box dimension is perhaps surprising to the non-expert.  That said, it is known that $1-H$ is, in fact, also an upper bound for the upper box dimension of all level sets, but this requires non-trivial ideas from stochastic analysis to prove, see \cite{daw, xiao}.

\subsection{Tube null sets}

A set $X$ in $\mathbb{R}^2$ is called tube null if it can be covered efficiently by tubes. The property of being tube null is quite subtle and relates to several different problems, see \cite{aleksi, carbery, mattilaFourier, shmerkinsuomala}. Here  a \emph{tube} is the $\delta$-neighbourhood of a line and the width of the tube is $\delta$.  More precisely,  $X$ is \emph{tube null} if for all $\eps>0$ there exists a  countable collection of tubes $T_i$ of width $|T_i|$ such that 
\[
\sum_i |T_i| < \eps.
\]
A tube null set must have Lebesgue measure zero but it can have Hausdorff dimension 2. Any set with Hausdorff dimension $<1$ is tube null.  More generally, for $1 \leq k<d$ integers and $s \in (0,1]$, we say a set $X \subseteq \rd$ is \emph{$(s,k)$-tube null} if for all $\eps>0$ there exists a countable collection of affine subspaces $\{L^i\}_i$ of dimension $k$ and a collection of corresponding widths $\delta_i$ such that 
\[
X \subseteq \bigcup_i (L^i)_{\delta_i}
\]
and such that
\[
\sum_i \delta_i^{s(d-k)} < \eps.
\]
The traditional definition of a set in $\rd$ being tube null now coincides with it being $(1,1)$-tube null and being $(s,1)$-tube null with $s<1$ is a strictly stronger condition meaning that the collection of tubes is `even smaller'.  We are unaware of this refined notion (with $s<1$) appearing in the literature but it is so natural that we suspect it  has been considered before. We first prove some basic estimates concerning $(s,k)$-tube null sets.
\begin{thm}\label{basictube}Let $1 \leq k< d$ and $s \in (0,1]$. 
\begin{enumerate}
\item If $X$ is an  $(s,k)$-tube null set in $\mathbb{R}^d$, then
\[
\hd X \leq k+s(d-k).
\]
\item  If $\hd X <s(d-k)$ then it is $(s,k)$-tube null.
\item There exist  $(s,k)$-tube null sets $X$  in $\mathbb{R}^d$ with $\hd X = k+s(d-k)$.
\end{enumerate}
\end{thm}

\begin{proof}
\begin{enumerate}
\item Let $\eps>0$.  Without loss of generality assume $X$ is bounded.  If it is not, then decompose it as a countable union of bounded sets.  Since $X$ is    $(s,k)$-tube null, then there exists a countable collection of affine subspaces $\{L^i\}_i$ of dimension $k$ and a collection of corresponding widths $\delta_i$ such that $\{ (L^i)_{\delta_i}\}_i$ is a cover of $X$ and such that
\[
\sum_i \delta_i^{s(d-k)} < \eps.
\]
Then, each $ (L^i)_{\delta_i} \cap X$ may be covered by $\approx \delta_i^{-k}$ sets of diameter $\delta_i$.  This yields a cover of $X$ and we may bound the Hausdorff content of $X$ by
\[
\mathcal{H}^{s(d-k)+k}_\infty(X) \lesssim \sum_i \delta_i^{-k} \delta_i^{s(d-k)+k} = \sum_i   \delta_i^{s(d-k)} < \eps
\]
and therefore
\[
\hd X \leq k+s(d-k).
\]
\item   Let $V \in G(d,d-k)$ be an arbitrary $d-k$ dimensional subspace of $\rd$ and write $\pi_V$ for orthogonal projection onto $V$ and $V^\perp$ for the orthogonal complement of $V$.  Since $\hd X <s(d-k) \leq d-k$, $\mathcal{H}_\infty^{s(d-k)}(\pi_V(X)) = 0$.  Let $\eps>0$ and $\{U_i\}_i$ be a cover of $\pi_V(X)$ such that
\[
\sum_i |U_i|^{s(d-k)} \leq \eps.
\]
Then $V^\perp \times U_i$ is contained in a `tube of width'  $|U_i|$ and the collection of such tubes covers $X$.  The $(s,k)$-tube nullity of $X$ follows.
\item Let $E \subseteq \mathbb{R}^{d-k}$ be a set with $\hd E = s(d-k)$ but $\mathcal{H}_\infty^{s(d-k)}(E) = 0$.  Let $X = [0,1]^k \times E$.  Immediately we get $\hd X = k+s(d-k)$.  Let $\eps>0$ and $\{U_i\}_i$ be a cover of $E$ such that
\[
\sum_i |U_i|^{s(d-k)} \leq \eps.
\]
Then $[0,1]^k \times U_i$ is contained in a tube of width  $|U_i|$ and the collection of such tubes covers $X$.  The $(s,k)$-tube nullity of $X$ follows.
\end{enumerate}
\end{proof}

Tube nullity forces large amounts of the set to be contained in tubes, which are (infinite extensions of) relatively flat cuboids.  We can therefore transfer information about tube nullity into our quantitative flatness framework and derive  information useful for Fourier restriction and dimension estimates.

\begin{thm} \label{tubenull}
Let $\mu$ be a compactly supported  measure  on a  $(s,k)$-tube null set $X$ in $\mathbb{R}^d$ and assume $s<1$. Then $\mu$ is $(s, d-k)$-flat. Therefore:
\begin{enumerate}
 \item If  the $E(p\to q)$ extension estimate 
\eqref{eq:extension} holds for $p \geq 1$ and $q \geq 2$, then
\[
\frac{p}{q(p-1)} \leq s
\]
In particular, if $p=2$, then 
\[
q \geq \frac{2}{s}.
\]
\item If $\mu$  satisfies the $I(p \to q)$ improving estimate \eqref{improvingest}, then
\[
\frac{1}{p}-\frac{1}{q} \leq  s. 
\] 

\item The bounds
\[
 \fd \mu \leq \fd X \leq   2 s (d-k)
\]
hold.
\end{enumerate}
\end{thm}

\begin{proof}
Let $\eps>0$ and  $\{L^i\}_i$,     $\delta_i$ witness the tube nullity of $X$ with respect to $\eps$.  Assume without loss of generality that $X \subseteq B(0,1)$.  Let $Q_i = (L^i)_{\delta_i} \cap B(0,1)$.  Then
\[
\mathcal{L}^d(Q_i) \approx \delta_i^{d-k}.
\]
Suppose
\[
\mu(Q_i) \lesssim  \mathcal{L}^d(Q_i)^{s} 
\]
holds for all $i$.  Then
\[
1 \leq \sum_i \mu(Q_i) \lesssim \sum_i \mathcal{L}^d(Q_i)^{s} \approx   \sum_i \delta_i^{s(d-k)}  < \eps
\]
which is a contradiction for all sufficiently small $\eps>0$.  It follows that there exists a sequence of cuboids $Q$ with volume  tending to zero such that 
\[
\mu(Q) \gtrsim  \mathcal{L}^d(Q)^{s}  
\]
and, therefore, $\mu$ is $s$-flat.  Moreover,  since $\delta_i$ is the associated `width' of $(L^i)_{\delta_i}$, it is also the shortest sidelength of $Q$ and so $\mu$ is also $(s,d-k)$-flat as required.  The results in (1) and (2) then follow from Theorems \ref{knapp} and \ref{thm:improving}, respectively.  We now turn our attention to  (3) and the Fourier dimension.  Note that the $(d-k)$ many shortest sides of $Q_i$ are all equal to $\delta_i$. Therefore, since $s(d-k) <d-k$, Theorem \ref{fourier3} gives
\[
\fd \mu \leq 2 s (d-k)
\]
and since this is true for all $\mu$ on $X$ we also get the same upper bound for $\fd X$, as required.
\end{proof}

Concerning the Fourier dimension bounds we were in the fortunate position above to be able to apply Theorem \ref{fourier3}.  However, we could also apply Corollary \ref{fourier}.  We get that if $\alpha$ is the Frostman dimension of $\mu$ then
\[
\fd \mu \leq \frac{2s(d-\alpha) }{1-s}.
\]
However, this bound is always worse than the bound $2s(d-k)$ obtained above.  This is because  by Theorem \ref{basictube} $\alpha \leq \hd X \leq k+s(d-k)$.

Next we revisit a some simple examples in the tube nullity context.

\begin{example}
The circle  $S^{1} \subseteq \mathbb{R}^2$ is $(s, 1)$-tube null for all $s>1/2$ and this can be seen by considering 
\[
\approx \delta^{-\frac{1}{2}}
\]
many uniformly spaced tangent lines to $S^{1}$.  In particular, the associated $\delta$-tubes cover $S^{1}$ and 
\[
\delta^{-\frac{1}{2}} \delta^s \to 0
\]
provided $s>1/2$.

Applying Theorem \ref{tubenull} implies that if the $E(p\to q)$ extension estimate 
\eqref{eq:extension} holds for any measure $\mu$ on $S^{1}$ for  $p \geq 1$ and $q \geq 2$, then
\[
\frac{p}{q(p-1)} \leq \frac{1}{2}
\]
In particular, if $p=2$, then 
\[
q \geq 4.
\]
 These estimates are not sharp and loss is down to the fact that we have considered diameter $\approx 1$ tubes instead of truncating them to $\approx \sqrt{\delta}$ as in Example \ref{sphereeg}. 

Applying Theorem \ref{tubenull} again we get 
\[
\fd (S^{1}) \leq  2s
\]
and letting $s\searrow 1/2$ gives
\[
\fd (S^{1}) \leq  1
\]
which is sharp.
\end{example}

The above approach does not work well for spheres in higher dimensions, since they are not tube null in a useful sense. Indeed,  consider 
\[
\approx \delta^{-\frac{d-1}{2}}
\]
many uniformly spaced tangent planes to $S^{d-1} \subseteq \rd$.  The associated $\delta$-tubes cover $S^{d-1}$ and  
\[
\delta^{-\frac{d-1}{2}} \delta^s \to 0
\]
provided $s>(d-1)/2$ but this is trivial for $d \geq 3$.

\begin{example}
The cone $C^{2} \subseteq \mathbb{R}^3$ is $(s, 2)$-tube null for all $s>1/2$ and this can be seen by considering 
\[
\approx \delta^{-\frac{1}{2}}
\]
many uniformly  spaced tangent planes to $C^{2}$.  The associated $\delta$-tubes cover $C^{2}$ and 
\[
\delta^{-\frac{1}{2}} \delta^s \to 0
\]
provided $s>1/2$.

Applying Theorem \ref{tubenull} implies that if   the $E(p\to q)$ extension estimate 
\eqref{eq:extension} holds for any measure $\mu$ on $C^2$ for  $p \geq 1$ and $q \geq 2$, then
\[
\frac{p}{q(p-1)} \leq \frac{1}{2}
\]
In particular, if $p=2$, then 
\[
q \geq 4.
\]
Again these  estimates are not sharp and the loss is again down to the fact that we have considered diameter $\approx 1$ tubes instead of truncating them to $\approx \sqrt{\delta}$ as in Example \ref{coneeg}. 

Applying Theorem \ref{tubenull} again we get 
\[
\fd (C^{2}) \leq  2s
\]
and letting $s\searrow 1/2$ gives
\[
\fd (C^{d-1}) \leq  1
\]
which \emph{is} sharp, and more interesting than the bound $\fd S^1 \leq 1$ since, in the case of the cone, the upper bound is strictly smaller than the Hausdorff dimension (which is 2).  This gives a simpler  and completely different proof of a result of Fraser, Harris and Kroon \cite{kroon}.
\end{example}

Again, the above strategy does not work in higher dimensions (this time $d \geq 4$) because then the cones fail to be tube null in a useful sense.  It does work for the `cone' in the plane, but this is a trivial example because it is just two line segments.

\begin{example}
Let $X=E \times [0,1]$ where $E$ is a Salem set of Hausdorff (and box) dimension $\beta \in (0,1)$.  Then $X$ is $(\beta, 1)$-tube null and so, by Theorem \ref{tubenull}, if   the $E(p\to q)$ extension estimate 
\eqref{eq:extension} holds for any measure $\mu$ on $X$ for  $p \geq 1$ and $q \geq 2$, then
\[
\frac{p}{q(p-1)} \leq \beta.
\]
In particular, if $p=2$, then 
\[
q \geq \frac{2}{\beta}
\]
which beats the estimate
\[
q \geq \frac{2}{1+\beta}
\]
coming from Theorem \ref{basic}. Further, applying Theorem \ref{tubenull} again,
\[
\fd X \leq 2 \beta.
\]
This is not sharp, and the  real answer is $\fd X = \beta$.  The lower bound follows by decomposing appropriate product measures on $X$ and the upper bound comes from the fact that projection cannot decrease Fourier dimension.
\end{example}

\subsection{Flatness via large slices}

An important problem in geometric measure theory is to understand the `slices' of a fractal and, in particular, how one can relate the dimensions of the slices to the dimensions of the original set.  Here a \emph{slice} is the intersection of the set with an affine subspace.  We write $A(d,k)$ for the affine Grassmannian consisting of all $k$-dimensional  affine subspaces $V= V'+z$ for some $V' \in G(d,k)$ and $z \in V^{\perp}$.

If a set has a large slice, then one might expect lots of the set to be concentrated near the associated affine subspace and that this provides quantitative flatness. However, this is not  the whole story as having a large slice is not enough to control the measure \emph{near} the slice.  Instead, the strategy is to  couple the information coming from the slice (which we quantify using the Assouad spectrum) with global control of the measure on small balls coming from the box dimension (of the measure) to get what we need.

\begin{thm} \label{slices}
Let $X \subseteq \rd$ be a compact set and $\mu$ be a measure fully supported on $X$. Fix a $k$-dimensional affine subspace $V \in A(d,k)$. Then $\mu$ is $(\gamma, \upsilon)$-flat for all 
\begin{equation} \label{expforg}
\gamma> \frac{\ubd \mu -(1-\theta)\as \left(V \cap X\right) }{k\theta+(d-k)}
\end{equation}
and
\[
\upsilon = k\theta+(d-k).
\]
Therefore:
\begin{enumerate}
\item If $\mu$   satisfies the $E(p \to q )$  extension estimate \eqref{eq:extension}, then for all $V \in A(d,k)$,
\[
\frac{p}{q(p-1)} \leq \frac{\ubd \mu -(1-\theta)\as \left(V \cap X\right) }{k\theta+(d-k)}
\]
In particular, if $p=2$ then 
\[
q \geq \frac{2(k\theta+(d-k))}{\ubd \mu -(1-\theta)\as \left(V \cap X\right) }.
\]
Setting $\theta=0$ we get that 
\[
\frac{p}{q(p-1)} \leq  \frac{\ubd \mu -\ubd \left(V \cap X\right) }{d-k}
\]
and, for $p=2$, 
\[
q \geq \frac{2(d-k)}{\ubd \mu -\ubd \left(V \cap X\right) }.
\]
\item If $\mu$  satisfies the $I(p \to q)$ improving estimate \eqref{improvingest}, then
\[
\frac{1}{p}-\frac{1}{q} \leq \frac{\ubd \mu -(1-\theta)\as \left(V \cap X\right) }{k\theta+(d-k)} . 
\] 
\item The bounds
\[
\fd \mu \leq \frac{2\gamma(d-\alpha)}{1-\gamma}
\]
where $\alpha$ is the Frostman dimension of $\mu$ and $\gamma$ can be taken as the right hand side of \eqref{expforg}. Moreover, provided 
\[
\ubd \mu -\as \left(V \cap X\right)(1-\theta)  < d-k
\]
we also get
\[
\fd \mu \leq 2 \left(\ubd \mu -\as \left(V \cap X\right)(1-\theta) \right).
\]
\end{enumerate}
\end{thm}
\begin{proof}
Let $V \in A(d,k)$  and let
\[
s<\as V \cap X.
\]
By definition there exists a sequence of scales $r \to 0$ and associated  cubes $Q' \subseteq V$ with sidelengths $r^\theta$ such that we may find an $r$-separated subset $E$ of $V \cap X \cap Q'$ of size at least $\gtrsim r^{s(\theta-1)}$.  Let $Q$ be the cuboid given by the $r$-thickening of $Q'$, that is, 
\[
Q = Q'_r \subseteq \rd.
\]
Then
\[
\mathcal{L}^{d}(Q) \approx   r^{k\theta+(d-k)}.
\]
Moreover, we may find a collection of disjoint $r$-balls centred at each point of $E$ which lie completely inside $Q$.  Let $t> \ubd \mu$.  By assumption, the $\mu$ mass of each of these $r$-balls is at least 
\[
\gtrsim r^{t}
\]
and therefore
\[
\mu(Q)  \gtrsim r^{s(\theta-1)+t}  = \mathcal{L}^{d}(Q)^{\frac{s(\theta-1)+t}{k\theta+(d-k)}}
\]
and so $\mu$ is $(\gamma, \upsilon)$-flat for 
\[
\gamma = \frac{s(\theta-1)+t}{k\theta+(d-k)}
\]
and  $\upsilon= k\theta+(d-k)$.  Letting $s \nearrow \as V\cap X$ and $t \searrow \ubd \mu$ gives the desired flatness result.

The estimates in (1) and (2) then follow from Theorems \ref{knapp} and \ref{thm:improving}, respectively.  We now turn our attention to the Fourier dimension and (3).  Applying Corollary \ref{fourier} we get the first bound.  To get the final estimate we need to be in a position to apply Theorem \ref{fourier3}.  This requires the additional assumption.  The $(d-k)$ many shortest sides of $Q$ are of length $\approx r$ and   so, provided $t$ and $s$ are such that
\[
t-s(1-\theta)  < d-k
\]
(which we can ensure by assumption),  we may apply Theorem \ref{fourier3} with $\upsilon= k\theta+(d-k)$ to obtain
\[
\fd \mu \leq 2 \frac{t-s(1-\theta) }{k\theta+(d-k)} \upsilon   = 2(t-s(1-\theta)) 
\]
and letting $s \nearrow \as V \cap X$ and $t \searrow \ubd \mu$ completes the proof of the final claim.
\end{proof}

Given the previous result, it is natural to ask whether non-trivial estimates come from Marstrand's slicing theorem which ensures that all sufficiently large sets must have some (in fact many) large slices.  However, the answer is no and, in fact, this is to be expected since the bounds would apply too generally.  In other words, to get non-trivial information we need atypically large slices.

In order to obtain strong results we are looking for a measure with small box dimension supported on a set with a slice with large Assouad spectrum.  In the extreme case when the former is as small as possible and the latter is as large as possible we obtain   optimal estimates.  Recall that we may always choose $\mu$ on $X$ with $\ubd \mu = \ubd X  $ and so if a set has a slice  with maximal Assouad spectrum then there always exists a measure---a very natural measure---on the set with Fourier dimension 0 and no restriction estimates. The following result is a direct corollary of Theorem \ref{slices} and Theorem \ref{flattening}.

\begin{cor} \label{slicegive0}
Suppose $X \subseteq \rd$ is compact and $\mu$ is a measure fully supported on $X$ with $\ubd \mu = \ubd X  $.  Suppose there exists    $V \in A(d,k)$ for some $1 \leq k < d$  and $\theta \in (0,1)$ such that 
\[
\as (V \cap X) = \frac{\ubd (V \cap X)}{1-\theta} = \frac{\ubd X}{1-\theta}.
\]
Then $\mu$ is $\gamma$-flat for all $\gamma>0$,  there are no non-trivial  $E(p \mapsto q)$ extension estimates   or   $I(p \to q)$ improving estimates  for $\mu$,   and $\mu$ is not $L^2$-flattening. Moreover, $\fd \mu = 0$, the Frostman and Sobolev dimensions of $\mu$ are at most $d-1/2$ and 
\[
\fs \mu \leq  (d-1/2) \, \theta 
\]
for all $0 \leq \theta \leq 1$.
\end{cor}

For certain  examples we may find that  atypically large slice exists by construction. We state a simple corollary which applies to two well-known dynamically defined fractals.

\begin{cor} \label{rauzy}
Suppose $X\subseteq \rd$ is a  compact set which contains a piece of a $k$-dimensional affine subspace.  Let $\mu$ be a measure fully supported on $X$  with upper box dimension $\ubd \mu = s$. Then $\mu$ is $(\gamma, (d-k))$-flat for all 
\begin{equation} \label{expforg2}
\gamma> \frac{s -k}{d-k}.
\end{equation}
Therefore:
\begin{enumerate}
\item If $\mu$   satisfies the $E(p \to q )$  extension estimate \eqref{eq:extension}, then 
\[
\frac{p}{q(p-1)} \leq \frac{s -k}{d-k}.
\]
In particular, if $p=2$ then 
\[
q \geq \frac{2(d-k)}{s -k}.
\]
\item If $\mu$  satisfies the $I(p \to q)$ improving estimate \eqref{improvingest}, then
\[
\frac{1}{p}-\frac{1}{q} \leq \frac{s -k}{d-k} . 
\] 
\item Provided  $s<d$, 
\[
\fd \mu \leq 2(s-k).
\]
\end{enumerate}
In particular,   any fully supported measure with $\ubd \mu <\frac{ \hd \mu}{2} +k$ is not Salem.   Further, any fully supported measure with $\ubd \mu = k$ has Fourier dimension zero,  no non-trivial extension or improving estimates, and is not $L^2$-flattening.

These estimates apply to the Sierpi\'nski triangle, Sierpi\'nski carpet, and the Rauzy gasket with $d=2$, and  $k=1$. 
\end{cor}

\subsection{Nonlinear slices}

The results in the previous section say that if an affine subspace intersects the support of a measure in a set with large Assouad spectrum then the measure is forced to be quantitatively flat.  A natural extension of this is to replace the affine subspace with a (nonlinear) $k$-dimensional surface with suitable smoothness.  This is not just a routine extension because the same results clearly cannot hold.  If the surface is curved enough then this can preserve Fourier decay even for measures which are very close to (or even completely contained in) the surface.  That is, we will necessarily see different features appearing in the estimates coming from the (potential) curvature. We write $a \wedge b = \min\{a,b\}$.

\begin{thm} \label{nonlinearslices}
Suppose $X \subseteq \rd$ is compact and $\mu$ is a measure fully supported on $X$. Fix a $C^2$ $k$-dimensional surface  $S \subseteq \rd$. Then $\mu$ is $(\gamma, \upsilon)$-flat for all 
\begin{equation} \label{expforg}
\gamma>  \frac{\ubd \mu -(1-\theta) \as (X \cap S)}{k\theta+(d-k) (2\theta \wedge 1)}
\end{equation}
and
\[
\upsilon = \frac{k\theta+(d-k) (2\theta \wedge 1)}{(2\theta \wedge 1)}.
\]
Therefore:
\begin{enumerate}
\item If $\mu$   satisfies the $E(p \to q )$  extension estimate \eqref{eq:extension}, then 
\[
\frac{p}{q(p-1)} \leq  \frac{\ubd \mu -(1-\theta) \as (X \cap S)}{k\theta+(d-k) (2\theta \wedge 1)}.
\]
In particular, if $p=2$, then 
\[
q \geq \frac{2(k\theta+(d-k) (2\theta \wedge 1))}{\ubd \mu -(1-\theta)\as \left(X \cap S\right) }.
\]
\item If $\mu$  satisfies the $I(p \to q)$ improving estimate \eqref{improvingest}, then
\[
\frac{1}{p}-\frac{1}{q} \leq   \frac{\ubd \mu -(1-\theta) \as (X \cap S)}{k\theta+(d-k) (2\theta \wedge 1)}. 
\] 
\item The bound
\[
\fd \mu \leq \frac{2\gamma(d-\alpha)}{1-\gamma}
\]
holds where $\alpha$ is the Frostman dimension of $\mu$ and $\gamma$ can be taken as the right hand side of \eqref{expforg}. Moreover, provided 
\[
\ubd \mu -(1-\theta)\as \left(X \cap S\right)  < (d-k)(2\theta \wedge 1)
\]
we get
\[
\fd \mu \leq \frac{2 \left(\ubd \mu -\as \left(S \cap X\right)(1-\theta) \right)}{2\theta \wedge 1}.
\]
\end{enumerate}
\end{thm}
\begin{proof}
Let
\[
s<\as (S \cap X).
\]
By definition there exists a sequence of scales $r \to 0$ and associated  balls $B \subseteq \rd$ of diameter $r^\theta$ centred in $S$ such that we may find an $r$-separated subset $E$ of $S \cap X \cap Q'$ of size at least $\gtrsim r^{s(\theta-1)}$.  Then we may find a collection of disjoint $r$-balls centred at each point of $E$ which lie completely inside $B$.  However, since $S$ is $C^2$ and of dimension $k$ we can find a cuboid $Q$ which contains all of the $r$-balls centred at points in $E$ which has $k$ sides of length $r^\theta$ and $(d-k)$ sides of length 
\[
\approx \max\{ r^{2\theta}, r\} = r^{2\theta \wedge 1}.
\]
This can be constructed by taking the ($k$-dimensional) tangent plane of $S$ at the centre of $B$ intersected with $B$, extending this to a $k$-dimensional square $Q'$ of sidelength $\approx r^\theta$ and then taking the $r^{2\theta \wedge 1}$ thickening of it, that is, 
\[
Q = Q'_{r^{2\theta \wedge 1}} \subseteq \rd.
\]
The smoothness of $S$ ensures that this thickening captures the desired balls. The $ r^{2\theta}$ captures all of $S\cap B$ and the $r$ term catches the $r$-balls centred in $E \subseteq S$. Then
\[
\mathcal{L}^{d}(Q) \approx   r^{k\theta+(d-k) (2\theta \wedge 1)}.
\]
 Let $t> \ubd \mu$.  By assumption, the $\mu$ mass of each of these $r$-balls is at least 
\[
\gtrsim r^{t}
\]
and therefore
\[
\mu(Q)  \gtrsim r^{s(\theta-1)+t}  = \mathcal{L}^{d}(Q)^{\frac{s(\theta-1)+t}{k\theta+(d-k) (2\theta \wedge 1)}}
\]
and so $\mu$ is $(\gamma, \upsilon)$-flat for 
\[
\gamma = \frac{s(\theta-1)+t}{k\theta+(d-k) (2\theta \wedge 1)}
\]
and 
\[
\upsilon= \frac{k\theta+(d-k) (2\theta \wedge 1)}{(2\theta \wedge 1)}
\]
noting that the shortest side of $Q$ has length $\approx r^{2\theta \wedge 1}$.  Letting $s \nearrow \as V\cap X$ and $t \searrow \ubd \mu$ gives the desired flatness result.

The results in (1) and (2) then follow from Theorems \ref{knapp} and \ref{thm:improving}, respectively.   We now turn our attention to the Fourier dimension and (3).  Applying Corollary \ref{fourier} we get the first bound.  To get the final estimate we need to be in a position to apply Theorem \ref{fourier3} and this requires the additional assumption.  The $(d-k)$ many shortest sides of $Q$ are of length $\approx r^{2\theta \wedge 1}$ and   so, provided $t$ and $s$ are such that
\[
t-s(1-\theta)  < d-k
\]
(which we can ensure by assumption), we may apply Theorem \ref{fourier3} with $\upsilon= k\theta+(d-k)$ to obtain
\[
\fd \mu \leq 2 \frac{t-s(1-\theta) }{k\theta+(d-k)} \upsilon   = 2(t-s(1-\theta)) 
\]
and letting $s \nearrow \as V \cap X$ and $t \searrow \ubd \mu$ completes the proof of the final claim.
\end{proof}

It is interesting to compare Theorem \ref{nonlinearslices} with Theorem \ref{slices}.  One notable difference is that Theorem \ref{nonlinearslices} does not give anything useful in the case $\theta=0$.  This is because when $\theta$ is very small, the (potential) curvature of $S$ forces $Q$ to become quantitatively less flat.  Moreover, we do not get a conclusion of the form: a measure with box dimension $k$ fully supported on a set which contains a piece of a $k$-dimensional surface has Fourier dimension zero.  But this is simply not true, for example the surface measure on a sphere. Nevertheless, we can still obtain an analogue of Corollary \ref{rauzy} but w are forced to use larger $\theta$.  In fact, the bound obtained is optimised at $\theta=1/2$ (provided $\ubd \mu \leq d$) but we see a loss compared with the estimate in Corollary \ref{rauzy} and this can be interpreted as coming from potential curvature.

\begin{cor}
Suppose $X\subseteq \rd$ is a  compact set which contains a piece of a $k$-dimensional $C^2$ surface.  Let $\mu$ be a measure fully supported on $X$  with upper box dimension $\ubd \mu = s$. Then $\mu$ is $(\gamma, d-k/2)$-flat for all 
\begin{equation} 
\gamma>  \frac{s -k/2}{d-k/2}.
\end{equation}
Therefore:
\begin{enumerate}
\item If $\mu$   satisfies the $E(p \to q )$  extension estimate \eqref{eq:extension}, then 
\[
\frac{p}{q(p-1)} \leq \frac{s -k/2}{d-k/2}.
\]
In particular, if $p=2$ then 
\[
q \geq \frac{2d-k}{s -k/2}.
\]
\item If $\mu$  satisfies the $I(p \to q)$ improving estimate \eqref{improvingest}, then
\[
\frac{1}{p}-\frac{1}{q} \leq  \frac{s -k/2}{d-k/2}. 
\] 

\item Provided 
\[
s-k/2 < d-k
\]
we get
\[
\fd \mu \leq 2s-k.
\]
\end{enumerate}
In particular, these estimates apply to the Apollonian circle packing with $d=2$, $k=1$ and $S$ a circle. 
\end{cor}

The bound $2s-k$ is sharp since for the Surface measure on the sphere $S^{d-1} \subseteq \rd$ and $S$ the sphere itself, $s=k=d-1$ and $\fd \mu = d-1$.

\subsection{Flatness via large projections}

Another central problem in geometric measure theory is to understand how the dimensions of fractal sets behave under orthogonal projection. The seminal Marstrand--Mattila projection theorem tells us that the projection of   a Borel set   $X \subseteq \rd$ onto a $k$-dimensional subspace has  Hausdorff dimension $\hd X \wedge k$ almost surely with respect to the Grassmannian measure. For $V \in G(d,k)$, we write $P_V$ for orthogonal projection of $\rd$ onto $V$.  

Having a projection with large Hausdorff dimension does not give us quantitative flatness but, on the other hand, a projection with large Assouad spectrum means that a large amount of the set has projected into a small ball and the preimage of this small ball under the projection is a thin tube.  This information, one again coupled with information coming from the box dimension of the relevant measure which ensures that all balls carry a certain amount of mass, forces a lot of mass into a  small tube and this is the quantitative flatness we need.  This idea is made precise in the following result.

\begin{thm} \label{projections}
Suppose $X \subseteq \rd$ is compact and $\mu$ is a measure supported on $X$ and let $V \in G(d,k)$ for some $1 \leq k < d$. Let $\theta \in [0,1)$ be arbitrary.  Then $\mu$ is $(\gamma, k)$-flat for all
\[
\gamma > \frac{\ubd \mu -(1-\theta)\as P_V(X)}{k\theta}.
\]
Therefore:
\begin{enumerate}
\item If  the $E(p\to q)$ extension estimate 
\eqref{eq:extension} holds, then
\[
\frac{p}{q(p-1)} \leq \frac{\ubd \mu -(1-\theta)\as P_V(X)}{k\theta}.
\]
In particular, if $p=2$, then 
\[
q \geq \frac{2(k\theta)}{\ubd \mu -(1-\theta)\as P_V(X)}.
\]
\item If $\mu$  satisfies the $I(p \to q)$ improving estimate \eqref{improvingest}, then
\[
\frac{1}{p}-\frac{1}{q} \leq \frac{\ubd \mu -(1-\theta)\as P_V(X)}{k\theta} . 
\] 
\item If
\[
 \frac{\ubd \mu-\as P_V(X)(1-\theta) }{\theta} <k,
\]
then the estimate 
\[
\fd \mu \leq  \frac{2(\ubd \mu-\as P_V(X)(1-\theta) )}{\theta} 
\]
holds. 
\end{enumerate}
Since these estimates hold for all $V$ and $\theta$, in practise one would then optimise over $V$ and $\theta$ to obtain the best result. 
\end{thm}
\begin{proof}
We may assume without loss of generality that $X \subseteq B(0,1)$. Let $s< \as P_V(X)$ and then by definition there exists a sequence of scales $r \to 0$ and associated  cubes $Q' \subseteq V$ with sidelengths $r^\theta$ such that we may find a $2r$-separated subset $E'$ of $P_V(X) \cap Q'$ of size at least $\gtrsim r^{s(\theta-1)}$.  Let $Q$ be a cuboid of diameter  $\approx 1$ and volume 
\[
\mathcal{L}^{d}(Q) \approx   r^{k\theta}
\] 
which contains  the $r$-neighbourhood of the  truncated tube
\[
  P_V^{-1}(Q'_r) \cap B(0,1).
\]
For each point in $x \in E'$, choose a point in $P_V^{-1}(x) \cap X$ and denote this new set by $E$ noting that the points in $E$ are still $r$-separated since projection is $1$-Lipschitz.  We may find a collection of disjoint $r$-balls centred at each point of $E$ which lie completely inside $Q$.  Let $t> \ubd \mu$ and then the $\mu$ mass of each of these $r$-balls is at least 
\[
\gtrsim r^{t}
\]
and therefore
\[
\mu(Q)  \gtrsim r^{s(\theta-1)+t}  = \mathcal{L}^{d}(Q)^{\frac{s(\theta-1)+t}{k\theta}}.
\]
In particular, $\mu$ is $(\gamma, \upsilon)$-flat for 
\[
\gamma = \frac{s(\theta-1)+t}{k\theta}
\]
and $\upsilon = k$ and letting $s \nearrow \as P_V(X)$ and $t \searrow \ubd \mu$ gives the desired flatness estimate.  The estimates  in (1) and (2) then follow from Theorems \ref{knapp} and \ref{thm:improving}, respectively.

We now turn our attention to the Fourier dimension and (3).  For this we need to be in a position to apply Theorem \ref{fourier2} and for this we need the additional assumption from (3). The shortest $k$ sides of $Q$ are all  of length $\approx r^\theta$ and   so, provided
\[
\frac{t-s(1-\theta) }{k\theta} k = \frac{t-s(1-\theta) }{\theta} < k
\]
which we can ensure by assumption,  we may apply Theorem \ref{fourier2} to obtain
\[
\fd \mu \leq 2 \frac{t-s(1-\theta) }{k\theta} k   = \frac{2(t-s(1-\theta) )}{\theta} 
\]
and letting $s \nearrow \as P_V(X)$ and $t \searrow \ubd \mu$ completes the proof of (3).
\end{proof}

In order to obtain strong results we are looking for a measure with small box dimension supported on a set which has  a projection with large Assouad spectrum.  In the extreme case, when the former is as small as possible and the latter is as large as possible, we obtain   optimal estimates.  Recall that we may always choose $\mu$ fully supported on $X$ with $\ubd \mu = \ubd X  $ and so if a set has a projection with maximal Assouad spectrum then there always exists a very natural measure on the set with Fourier dimension 0 and no restriction estimates.

\begin{cor} \label{projgive0}
Suppose $X \subseteq \rd$ is compact and $\mu$ is a measure supported on $X$ with $\ubd \mu = \ubd X  $.  Suppose there exists    $V \in G(d,k)$ for some $1 \leq k < d$  and $\theta \in (0,1)$ such that 
\[
\as P_V(X) = \frac{\ubd P_V(X)}{1-\theta} = \frac{\ubd X}{1-\theta}.
\]
Then $\mu$ is $\gamma$-flat for all $\gamma>0$,  there are no non-trivial  extension or $L^p$-improving estimates   for $\mu$,   and $\mu$ is not $L^2$-flattening. Moreover, $\fd \mu = 0$, the Frostman and Sobolev dimensions of $\mu$ are at most $d-1/2$ and 
\[
\fs \mu \leq  (d-1/2) \, \theta 
\]
for all $0 \leq \theta \leq 1$.
\end{cor}

Next we present a modest but perhaps surprising application to a well-known problem concerning the dimensions of self-similar sets.  Recall that a self-similar set is a non-empty  set $X \subseteq \rd$ which is invariant under an iterated function system (IFS) of similarity maps.  See \cite{barany, Fal03} for more details on self-similar sets and their dimension theory.  It is known that if the open set condition is satisfied, then the Assouad, box and Hausdorff dimensions of $X$ coincide and are given by a simple closed formula (the Hutchinson--Moran formula). However, if the open set condition fails then there are two rather different regimes. If the weak separation condition holds, then these three dimensions coincide, but if the weak separation condition fails, then the Assouad dimension may be strictly larger.  Nevertheless, it is widely believed that without any separation assumptions the Assouad spectrum remains equal to the box dimension for all $\theta$. This is known for certain cases, for example, if the self-similar set is in the line, the defining IFS is algebraic, and there are no exact overlaps.  A modest step in the direction of this conjecture would be to bound the Assouad spectrum above by something strictly smaller than the general upper bound coming from the box dimension, that is, 
\[
\as X \leq \frac{\ubd X}{1-\theta}.
\]
We achieve this using our quantitative flatness assumption combined with known results about the $L^2$-flattening nature  of self-similar measures.  The result applies to self-similar sets whose Hausdorff dimension is given by the Hutchinson--Moran formula.  The interesting case is when the open set condition fails and so we are interested in the situation when there are no exact overlaps and the exact overlaps conjecture holds. We can handle more general classes than this with the following method but we leave the details to the reader.  The important thing is that the dimension is obtained as the sequence of solutions to Hutchinson--Moran formulae applied to higher iterates of the defining IFS which account for exact overlaps.

The result below can be obtained   by more direct methods, see \cite{jon:book}. We include the result  because its  proof  involves a simple and novel synthesis of modern tools.

\begin{cor}
Let $X \subseteq \rd$ be a self-similar set whose Hausdorff dimension is given by the Hutchinson--Moran formula.  Then for all $\theta \in (0,1)$
\[
\as X < \frac{\ubd X}{1-\theta}.
\]
\end{cor}

\begin{proof}
Suppose the conclusion is not satisfied, and therefore we may find $\theta \in (0,1)$ such that 
\[
\as X = \frac{\ubd X}{1-\theta}.
\]
Let $Y \subseteq \rd \times \mathbb{R}^n$  be a self-similar lift of $X$ satisfying the open set condition and such that $Y$ is not contained in any hyperplane. Since the dimension of $X$ is given by the Hutchinson--Moran formula, we can ensure that the dimension of $Y$ is given by the same formula and therefore $\ubd X = \ubd Y$.   Let $\mu$ be the natural self-similar measure on $Y$ which is Ahlfors regular, in particular, having $\ubd \mu = \ubd X$. Moreover, by construction the projection $P$ onto the first $d$ coordinates satisfies
\[
\as P(Y) = \as X = \frac{\ubd X}{1-\theta} = \frac{\ubd Y}{1-\theta}  .
\]
  Then, by Corollary \ref{projgive0}, $\mu$ is $\gamma$-flat for all $\gamma>0$ and so by Theorem \ref{flattening}
\[
\fs \mu \leq (d+n-1/2) \theta
\]
for all $\theta \in [0,1]$. However, this contradicts the fact that $\mu$ is uniformly affinely non-concentrated in the sense of Khalil \cite{khalil}.  This follows by direct calculation relying on the fact that it is not supported in a line and it is doubling (even Ahlfors regular).  In particular, being uniformly affinely non-concentrated ensures that the Sobolev dimension of the iterated convolutions $\mu^{\ast k}$ converges to the ambient spatial dimension $d+n$ \cite{khalil}.  This, combined with \cite[Theorem 6.1]{Fra24}, this yields
\[
\liminf \frac{\fs \mu}{(d+n) \theta} \geq 1  .
\]
This completes the proof.
\end{proof}

\section{Examples and specific constructions}

\subsection{Self-affine sets}

One of the most important ways to generate fractals sets, especially with a dynamical flavour, is via iterated function systems (IFSs).  An IFS is a (usually finite) collection of contractions acting on  a common compact set $K \subseteq \rd$. Such a collection defines a unique attractor which is often fractal.  Here we focus on the case where the contractions are affine maps $\{x \mapsto A_i(x) + t_i\}_{i\in \mathcal{I}}$ with each $A_i$ an invertible matrix and each $t_i$ a translation. We write $S_i$ for the map $S_i(x) = A_i(x)+t_i$.   The associated \emph{attractor} or \emph{self-affine set}  is the unique non-empty compact set $X \subseteq \rd$ satisfying
\[
X = \bigcup_{i \in \mathcal{I}} \left(A_i(X)+t_i\right).
\]
There is a natural surjection from the shift space $\mathcal{I}^\mathbb{N}$  to the attractor (the coding map) and the dynamics generated by the IFS is closely related to the dynamics on the shift space given by the left shift.  We consider (in the first instance not necessarily invariant) measures $\nu$ on the shift space and their pushforwards onto $X$ which we denote by $\mu$.  These are our measures of interest.  For example, if $\nu$ is a Bernoulli measure, then the associated $\mu$ is a self-affine measure.  We refer the reader to \cite{barany,Fal03} for more background on self-affine sets and measures.

We use standard symbolic notation to describe the $\mu$.  We write $\mathcal{I}^*$ for the set of all finite words over $\mathcal{I}$ and for $\mathbf{i}=(i_1, i_2, \dots, i_k)   \in \mathcal{I}^*$ we write $1 \geq \alpha_1(\mathbf{i}) \geq  \alpha_2(\mathbf{i})  \geq \cdots \geq  \alpha_d(\mathbf{i}) >0$ for the \emph{singular values} of the matrix $A_\mathbf{i} = A_{i_1} \circ \cdots \circ  A_{i_k}$.  We also write $[\mathbf{i}] \subseteq  \mathcal{I}^\mathbb{N}$ for the cylinder consisting of all infinite words beginning with $\mathbf{i}$.  Note that the image of $[\mathbf{i}]$ under the  coding map in $X$ is contained in the geometric cylinder $S_\mathbf{i}(X)$ where $S_\mathbf{i} = S_{i_1} \circ \cdots \circ S_{i_k}$. Once again, we write $a \wedge b = \min\{a,b\}$.

In most situations, the geometric cylinders $S_\mathbf{i}(X)$ become very stretched and flattened as the length of $\mathbf{i}$ increases.  We use this to force quantitative flatness of the relevant measure.

\begin{thm} \label{affine}
Let $\mu$ be the pushforward  of any measure $\nu$ on the shift space onto a self-affine set.  Let $\mathbf{i} \in \mathcal{I}^\mathbb{N}$ be arbitrary. Then $\mu$ is $\gamma$-flat for all
\[
\gamma> \liminf_{k \to \infty}    \frac{\log \nu([\mathbf{i}_k])}{\log \textup{det}(A_{\mathbf{i}_k})} 
\] 
 where   $\mathbf{i}_k \in \mathcal{I}^k$ denotes the  word consisting of the first $k$-entries of $\mathbf{i}$.   Moreover, $\mu$ is $(\gamma,\upsilon)$-flat for $\gamma$ and $\upsilon$ with the product $\upsilon \gamma$ larger than but arbitrarily close to
\[
 \liminf_{k \to \infty} \frac{\log \nu([\mathbf{i}_k])}{\log \alpha_d(\mathbf{i}_k)}.
\]
Therefore:
\begin{enumerate}
\item If the   $E(p\to q)$ extension estimate 
\eqref{eq:extension} holds for $\mu$, then
\[
\frac{p}{q(p-1)} \leq  \liminf_{k \to \infty}   \frac{\log \nu([\mathbf{i}_k])}{\log \textup{det}(A_{\mathbf{i}_k})} 
\]
and, for $p=2$, 
\[
q \geq 2 \limsup_{k \to \infty}    \frac{\log \textup{det}(A_{\mathbf{i}_k})} {\log \nu([\mathbf{i}_k])} .
\]
\item If $\mu$  satisfies the $I(p \to q)$ improving estimate \eqref{improvingest}, then
\[
\frac{1}{p}-\frac{1}{q} \leq  \liminf_{k \to \infty}   \frac{\log \nu([\mathbf{i}_k])}{\log \textup{det}(A_{\mathbf{i}_k})}. 
\] 
\item Provided 
\[
 \liminf_{k \to \infty} \frac{\log \nu([\mathbf{i}_k])}{\log \alpha_d(\mathbf{i}_k)}  <1
\]
then
\[
\fd \mu \leq 2 \liminf_{k \to \infty} \frac{\log \nu([\mathbf{i}_k])}{\log \alpha_d(\mathbf{i}_k)}  .
\]
\end{enumerate}
\end{thm}

\begin{proof}
We may assume without loss of generality that each map in the defining IFS maps the unit ball $B(0,1)$ into itself, and thus $X \subseteq B(0,1)$.  Let $k$ be large and such that 
\[
\gamma> \frac{\log \nu([\mathbf{i}_k])}{\log \textup{det}(A_{\mathbf{i}_k})} . 
\]
Let $Q \subseteq \rd$ be the natural cube of sidelengths given by twice the singular values of $A_{\mathbf{i}_k}$ containing the geometric cylinder $S_{\mathbf{i}_k}(B(0,1))$.  Then
\[
\mathcal{L}^d(Q) = 2^d \alpha_1(\mathbf{i}_k) \alpha_2(\mathbf{i}_k)   \cdots \alpha_d(\mathbf{i}_k)  = 2^d \textup{det}(A_{\mathbf{i}_k})
\]
and 
\[
\mu(Q) \geq  \mu(S_{\mathbf{i}_k}(B(0,1)) ) \geq  \nu([\mathbf{i}_k]) =  \mathcal{L}^d(Q) ^{\frac{\log \nu([\mathbf{i}_k])}{\log (2^d\textup{det}(A_{\mathbf{i}_k}))}} \gtrsim \mathcal{L}^d(Q) ^{\gamma}.
\]
The desired $\gamma$-flatness follows and the estimates in (1) and (2) then follow from Theorems \ref{knapp} and \ref{thm:improving}, respectively.

Turning our attention to $(\gamma,\upsilon)$-flatness and the estimates in (3), let   $k$ be large and such that 
\[
\upsilon \gamma>\frac{\log \nu([\mathbf{i}_k])}{\log \alpha_d(\mathbf{i}_k)} =\frac{\log \nu([\mathbf{i}_k])}{\log \textup{det}(A_{\mathbf{i}_k})} \frac{\log \textup{det}(A_{\mathbf{i}_k})}{\log \alpha_d(\mathbf{i}_k)}  
\]
and we may assume that 
\[
  \gamma> \frac{\log \nu([\mathbf{i}_k])}{\log \textup{det}(A_{\mathbf{i}_k})}  
\]
and
\[
\upsilon  >  \frac{\log \textup{det}(A_{\mathbf{i}_k})}{\log \alpha_d(\mathbf{i}_k)}  .
\]
Constructing $Q$ as above, $\mu$ is $\gamma$-flat and, moreover, since the shortest side of $Q$ is of length $\approx \alpha_d(\mathbf{i})$  and
\[
\mathcal{L}^d(Q) =   2^d \textup{det}(A_{\mathbf{i}_k}) =
\alpha_d(\mathbf{i})^{\frac{\log (2^d\textup{det}(A_{\mathbf{i}_k}))}{\log \alpha_d(\mathbf{i})}} \gtrsim \alpha_d(\mathbf{i})^{\upsilon}
\]
the desired  $(\gamma,\upsilon)$-flatness follows.   The Fourier dimension bound in (3) then follows from Theorem \ref{fourier2}.  We are forced to apply this in the case $m=1$ and for this we need $\upsilon \gamma<1$, however, this can be ensured by the additional assumption.
\end{proof}

The previous theorem was rather general since it concerned any measure $\nu$ and we were forced to carefully optimise the product $\upsilon \gamma$ instead of optimising $\gamma$ and $\upsilon$ separately.  Next we provide some more digestible statements.  If the measure $\nu$ is ergodic, then we know the $\nu$-typical behaviour of the relevant limits in the above result.  This allows us to give cleaner bounds in terms of well-known quantities, although these bounds may not be the best possible.  

\begin{cor}\label{ergodic}
Let $\mu$ be the projection of an ergodic measure $\nu$ on the shift space onto a self-affine set.  Then $\mu$ is $(\gamma, \upsilon)$-flat for all
\[
\gamma >\frac{H(\mu)}{\chi_1(\mu)  \cdots  \chi_d(\mu)}
\]
and
\[
\upsilon >  \frac{\chi_1(\mu)  \cdots  \chi_d(\mu)} {\chi_d(\mu)}.
\]
where $H(\mu) \geq 0$ is the Shannon entropy of $\nu$ and $1 \leq \chi_1(\mu) \leq  \cdots \leq  \chi_d(\mu)$ are the Lyapunov exponents of $\mu$.   Therefore:
\begin{enumerate}
\item If the   $E(p\to q)$ extension estimate \eqref{eq:extension} holds for $\mu$, then
\[
\frac{p}{q(p-1)} \leq  \frac{H(\mu)}{\chi_1(\mu)  \cdots  \chi_d(\mu)}
\]
and, for $p=2$, 
\[
q \geq \frac{2 \chi_1(\mu)  \cdots  \chi_d(\mu)}{H(\mu)}.
\]
\item If $\mu$  satisfies the $I(p \to q)$ improving estimate \eqref{improvingest}, then
\[
\frac{1}{p}-\frac{1}{q} \leq \frac{H(\mu)}{\chi_1(\mu)  \cdots  \chi_d(\mu)} . 
\] 
\item If $H(\mu)<\chi_d(\mu)$, then
\[
\fd \mu \leq 2 \frac{H(\mu)}{\chi_d(\mu)}.
\]
\end{enumerate}
\end{cor}

\begin{proof}
Since  $\nu$  is ergodic, by the Shannon--McMillan--Breiman theorem we know that for $\nu$ almost all $\mathbf{i} \in \mathcal{I}^\mathbb{N}$ the following limits exist
\[
-\lim_{k \to \infty} \frac{1}{k} \log \nu([\mathbf{i}_k]) = H(\nu)=: H(\mu)
\]
where $H(\nu)$ is the Shannon entropy of $\nu$ and $H(\mu)$ is the Shannon entropy of $\mu$. Moreover, by Oseledet's multiplicative ergodic theorem, for each $l = 1, \dots, d$,
\[
-\lim_{k \to \infty} \frac{1}{k} \log \alpha_l( \mathbf{i}_k) =  \chi_l(\mu).
\]
In particular, we may choose $\mathbf{i} \in \mathcal{I}^\mathbb{N}$  such that all of these conclusions hold simultaneously and the results follow. 
\end{proof}

The approach to bound the Fourier dimension of an ergodic measure on a self-affine set based on Theorem \ref{fourier3} works best in the plane ($d=2$) because in general all of the singular values or Lyapunov exponents of $\mu$ are different (that is, we cannot expect to do better than $m=1$).  For this reason we now specialise to the planar case and we consider the situation where we know the Hausdorff dimension of the measure. Higher dimensional analogues are also possible but we leave these to the reader.    Jordan--Pollicott--Simon \cite{jps} introduced the \emph{Lyapunov dimension} of $\mu$ in $\mathbb{R}^2$ as
\[
\frac{H(\mu)}{\chi_1(\mu)} \wedge \frac{H(\mu)+\chi_2(\mu)-\chi_1(\mu)}{\chi_2(\mu)}
\]
and proved that this is always an upper bound for the Hausdorff dimension of $\mu$.  However, in many cases the Hausdorff and Lyapunov dimensions coincide, see \cite[Proposition 7.2]{feng}.  

The following very special case of Theorem \ref{affine} and Corollary \ref{ergodic}  gives simple conditions in terms of familiar quantities which ensure that a self-affine measure is not Salem.

\begin{cor}
Let $\mu$ be an ergodic measure on a self-affine set and suppose $\hd \mu$ is given by the Lyapunov dimension.  
\begin{enumerate}
\item If $0<H(\mu) \leq \chi_1(\mu)$ and $2\chi_1(\mu) <\chi_2(\mu)$, then $\mu$ is not Salem.
\item If $ \chi_1(\mu)<H(\mu) < \chi_2(\mu)-\chi_1(\mu)$, then $\mu$ is not Salem.
\end{enumerate}
\end{cor}

\begin{proof}
By Corollary \ref{ergodic} and the fact that the $\hd \mu$ is given by the Lyapunov dimension, in case (1)
\[
\fd \mu \leq 2 \frac{H(\mu)}{\chi_2(\mu)} < \frac{H(\mu)}{\chi_1(\mu)} = \hd \mu
\]
and so $\mu$ is not Salem.  Further, in case (2), 
\[
\fd \mu \leq 2 \frac{H(\mu)}{\chi_2(\mu)} <\frac{H(\mu)+\chi_2(\mu)-\chi_1(\mu)}{\chi_2(\mu)} = \hd \mu
\]
and so $\mu$ is not Salem.
\end{proof}

\subsection{Parabolic Julia sets}

Let $T$ be a rational map of degree at least 2 acting on the Riemann sphere $\mathbb{C} \cup \{\infty\}$ and let $J_T$ be the closure of the repelling periodic points of $T$.  Provided $J_T$ is not the whole sphere, we may assume it is a compact subset of $\mathbb{C}$ (via conjugation by an appropriate M\"obius transformation which sends a point not in $J_T$ to the point at infinity) and so we may identify $J_T$ with a compact set in $\mathbb{R}^2$.  The set $J_T$ is called the \emph{Julia set} of $T$ and is often a beautiful fractal set.  Julia sets are one of the most well-studied and important classes of dynamically defined fractal. We refer the reader to \cite{beardonjulia,stratmannurbanski} for more background on the dynamics of rational maps and Julia sets.

 We specialise further to the case of \emph{parabolic Julia sets} which are those containing no critical points of $T$ but at least one parabolic point, where a \emph{parabolic point} is a fixed point $\omega \in \mathbb{C}$ for $T$ such that $|T'(\omega)|=1$.  Then $T$ admits an expansion 
\[
T(z) = z+a(z-\omega)^{p(\omega)+1}+ \dots
\]
in a neighbourhood of $\omega$ where $p(\omega) \in \mathbb{N}$ is the \emph{petal number} of $\omega$.  The petal numbers play an important role in the dynamics of $T$, especially near the associated  parabolic points $\omega$. The Hausdorff and box dimensions of a parabolic Julia set are given by an explicit pressure formula and we denote this value by $h$.  Note that
\[
\frac{p(\omega)}{1+p(\omega)}<h<2
\]  
for all parabolic points $\omega$, see \cite{adu}.  For a parabolic Julia set, there exist  natural $h$-conformal measures $\mu$ on $J_T$  and these are  our measure of interest, see \cite{du,stratmannurbanski}. Such a $\mu$ is a measure of maximal dimension, that is, $\hd \mu = h = \hd J_T$.

Leclerc \cite{leclerc} recently proved that hyperbolic Julia sets have positive Fourier dimension, but the precise value is non-quantitative and hard to pin down. 
We are able to use the dynamics  near parabolic points to force mass into long thin rectangles, thereby obtaining  quantitative flatness estimates.

\begin{thm} \label{julia}
Let $\mu$ be an  $h$-conformal  measure supported on a parabolic Julia set in $\mathbb{C}$ and let $\omega$ be a parabolic point.   Then $\mu$ is $(\gamma, \upsilon)$-flat for all
\[
\gamma > \frac{h+(h-1)p(\omega)}{2+p(\omega)}
\]
and
\[
\upsilon = \frac{2+p(\omega)}{1+p(\omega)}.
\]
Therefore:
\begin{enumerate}
\item If  the $E(p\to q)$ extension estimate 
\eqref{eq:extension} holds for $\mu$, then
\[
\frac{p}{q(p-1)} \leq \frac{h+(h-1)p(\omega)}{2+p(\omega)}
\]
and, for $p=2$, 
\[
q \geq \frac{4+2p(\omega)}{h+(h-1)p(\omega)} .
\]
\item If $\mu$  satisfies the $I(p \to q)$ improving estimate \eqref{improvingest}, then
\[
\frac{1}{p}-\frac{1}{q} \leq  \frac{h+(h-1)p(\omega)}{2+p(\omega)}. 
\] 
\item The bound
 \[
  \fd \mu \leq \frac{2h+2(h-1)p(\omega)}{1+p(\omega)} = 2h - \frac{2p(\omega)}{1+p(\omega)}
  \]
  holds. In particular, if $h<\frac{2p(\omega)}{1+p(\omega)}$, then $\mu$ is not Salem. 
\end{enumerate}
The bounds above are optimised by taking $\omega$ with maximal $p(\omega)$.
\end{thm}
\begin{proof}
Consider the Julia set in a neighbourhood of $\omega$.  The Fatou flower theorem tells us that  $J_T$ is squeezed between $p(\omega)$ many open `petals' tangent to each other along $p(\omega)$ many  rays emanating from  $\omega$.  This suggests that mass is concentrated near these tangent rays but we need two quantitative estimates to make this precise.  First, the global measure formula which describes the $\mu$ mass of small balls up to constants gives that
\[
 \mu(B(\omega,r)) \approx r^{h+(h-1)p(\omega)}.
\]
The global measure formula was proved by Stratmann and Urba\'nski \cite{stratmannurbanski}. Next, we need the quantitative Fatou flower theorem, proved in \cite[Lemma 5.5]{liam}, which says that 
\[
B(\omega,r) \cap J_T
\]
is contained in the $\approx r^{1+p(\omega)} |\log r |$ neighbourhood of the lines emanating from $\omega$.  This means that for all small $r$, we can find a rectangle $Q$ with sidelengths $\approx r$ and $\approx r^{1+p(\omega)} |\log r |$ oriented along one of the rays which carries mass
\[
\mu(Q) \approx  \mu(B(\omega,r)) \approx r^{h+(h-1)p(\omega)}.
\]
In particular, 
\[
\mathcal{L}^2(Q) \approx r^{2+p(\omega)}  |\log r | \geq \left(r^{1+p(\omega)}\right)^{\frac{2+p(\omega)}{1+p(\omega)}}
\]
and so
\[
\mu(Q) \gtrsim \mathcal{L}^2(Q)^{\gamma}
\]
for any fixed 
\[
\gamma > \frac{h+(h-1)p(\omega)}{2+p(\omega)}.
\]
The desired $(\gamma, \upsilon)$-flatness follows.  The estimates in (1) and (2) then follow from Theorems \ref{knapp} and \ref{thm:improving}, respectively. Turning our attention to Fourier dimension and the estimates in (3), we seek to apply Theorem \ref{fourier2}.  Then, since we can choose $\upsilon \gamma$  as close as we like to 
  \[
  \frac{h+(h-1)p(\omega)}{2+p(\omega)} \cdot   \frac{2+p(\omega)}{1+p(\omega)} = \frac{h+(h-1)p(\omega)}{1+p(\omega)} ,
  \]
  as long as this is strictly less than 1, that is, $h<\frac{1+2p(\omega)}{1+p(\omega)}$,    we can apply Theorem \ref{fourier2} to get
  \[
  \fd \mu \leq \frac{2h+2(h-1)p(\omega)}{1+p(\omega)}.
  \]
  However, the bound also holds in the case $h\geq \frac{1+2p(\omega)}{1+p(\omega)}$ but is not useful because it exceeds 2.  Indeed, since $J_T$ is porous, $\mu$ is not absolutely continuous and so we certainly have the bound $\fd \mu \leq 2$.
  
  Finally, we know $\hd \mu = h$ and so, comparing this with the upper bound for the Fourier dimension, we get that if $h<\frac{2p(\omega)}{1+p(\omega)}$, then $\fd \mu < \hd \mu$ and  $\mu$ is not Salem. 
  \end{proof}
  
   The upper bound for the Fourier dimension in Theorem \ref{julia} (3) approaches zero as $h$ approaches
  \[
  \max_\omega \frac{p(\omega}{1+p(\omega)}.
  \]
  This is a (strict) general lower bound for $h$ but it is sharp and so by taking $h$ arbitrarily close to this limit we can exhibit examples where the Fourier dimension is arbitrarily small while the Hausdorff dimension remains bounded away from zero. 
  
 The estimate in the case $p=2$ in Theorem \ref{julia} (1)  is always a non-trivial estimate since $h<2$ and so the right hand side is $>2$ but, moreover, it is always better than the bound 
\[
q \geq \frac{4}{\frd \mu}
\]
coming from Theorem \ref{basic}. This can be checked by using the additional information that 
\begin{equation} \label{frostjulia}
  \frd \mu = \min\left\{h,h+(h-1)\max_\omega p(\omega)\right\}
  \end{equation}
  which follows from the global measure formula.  
  
      We obtained the upper bound for the Fourier dimension in Theorem \ref{julia}  from Theorem \ref{fourier2} but we can also obtain an estimate from Corollary \ref{fourier} using the Frostman dimension formula \eqref{frostjulia} above. Interestingly, the alternative bound coming from Corollary \ref{fourier} is the same as the one we obtained in Theorem \ref{julia} in the case $h\geq1$ but for $h<1$ it is strictly worse.  

\subsection{Kleinian limit sets}

Let $\Gamma$ be a non-elementary geometrically finite Kleinian group acting on $(d+1)$-dimensional hyperbolic space with $d \geq 1$ an integer. Assume $\Gamma$ contains at least one parabolic element.  Let $L(\Gamma)$ denote the \emph{limit set} which is the closure of the hyperbolic fixed points of  $\Gamma$.  We assume that $L(\Gamma)$ is not the whole boundary and therefore we may assume that $L(\Gamma)$ is a compact subset of  $\rd$, that is, we use the upper half space model of hyperbolic space. The limit set is often a beautiful fractal set which encodes many dynamical features of the action of $\Gamma$.  We refer the reader to \cite{beardonkleinian,bowditch, stratmannvelani} for more background on Kleinian group actions and limit sets.

An element of $\Gamma$ is called \emph{parabolic} if it fixes precisely one point on the boundary of hyperbolic space and we call this fixed point a \emph{parabolic point}. The parabolic points are necessarily in $L(\Gamma)$.   For a parabolic point $z \in L(\Gamma)$, let $k(z) \in \{1, \dots, d\}$ be the \emph{parabolic rank} of $z$ which is the rank of the largest free Abelian subgroup of the stabiliser of $z$.    The Hausdorff and box dimensions of $L(\Gamma)$ are given by the \emph{Poincar\'e exponent} $\delta$ which is the critical exponent of the Poincar\'e series.  Note that
\[
\frac{k(z)}{2} < \delta < d
\]
for all parabolic points $z$, see \cite{tukia}. The second inequality holds because we are in the geometrically finite case and we assumed the limit set was not the whole boundary. The natural measure supported on the limit set is the \emph{Patterson--Sullivan measure} $\mu$ and this is our measure of interest. This is a $\delta$-conformal measure of maximal dimension $\hd \mu = \delta = \hd L(\Gamma)$ and is intimately connected to the action of $\Gamma$ and the geometry of $L(\Gamma)$.

Bourgain and Dyatlov \cite{bourgaindyatlov} and later Li, Naud and Pan \cite{naud} proved that the Patterson--Sullivan measure enjoys polynomial Fourier decay in certain situations.  In particular, it has positive Fourier dimension, but the precise value is non-quantitative and hard to pin down.  As such, it is useful to provide concrete upper bounds and we do this here. We are able to use the local structure of the limit set near parabolic points to force mass into flat slabs, and use this to  obtain  quantitative flatness estimates.

\begin{thm} \label{kleinian}
Let $\mu$ be the \emph{Patterson--Sullivan} measure supported on a non-elementary geometrically finite Kleinian group acting on $(d+1)$-dimensional hyperbolic space and let $z \in L(\Gamma)$ be a parabolic point.  Then $\mu$ is $(\gamma, \upsilon)$-flat for
\[
\gamma = \frac{2\delta-k(z)}{2d-k(z)}
\]
and
\[
\upsilon = d-k(z)/2.
\]
Therefore:
\begin{enumerate}
\item If  the $E(p\to q)$ extension estimate 
\eqref{eq:extension} holds for $\mu$, then
\[
\frac{p}{q(p-1)} \leq \frac{2\delta-k(z)}{2d-k(z)},
\]
and, for $p=2$, 
\[
q \geq \frac{4d-2k(z)}{2\delta-k(z)} .
\]
\item If $\mu$  satisfies the $I(p \to q)$ improving estimate \eqref{improvingest}, then
\[
\frac{1}{p}-\frac{1}{q} \leq  \frac{2\delta-k(z)}{2d-k(z)}. 
\] 
\item The bound
  \[
   \fd \mu \leq \max\left\{2\delta - k(z) , \frac{(2\delta - k(z))(d-2\delta+k(z))}{d-\delta}\right\}
  \]
  holds. In particular, if $\delta<k(z)\wedge\frac{d+k(z)}{3}$, then $\mu$ is not Salem. Further, provided $\delta<d-k(z)/2$, the bound
 \[
  \fd \mu \leq 2\delta - k(z)
  \]
  also holds.  In particular, if $\delta<k(z)\wedge(d-k(z)/2)$, then $\mu$ is not Salem. 
\end{enumerate}
The bounds  above are optimised by taking $z$ with maximal $k(z)$.
\end{thm}

\begin{proof}
Consider the limit set in a neighbourhood of $z$.  The global measure formula which describes the $\mu$ mass of small balls up to constants tells us that
\[
  \mu(B(z,r)) \approx r^{2\delta-k(z)}.
\]
The global measure formula was  proved by Stratmann and Velani \cite{stratmannvelani}. Now, we seek to squeeze this mass into a relatively flat part of the space.  Briefly conjugate $z$ to $\infty$ via a M\"obius transformation.  Then, since $\Gamma$ is geometrically finite, one of the many characterisations of geometric finiteness (see \cite{bowditch})  ensures that  the limit set is contained in a $\lesssim 1$ neighbourhood of a $k(z)$-dimensional linear subspace $V$  and conjugating back we get  that 
\[
L(\Gamma) \cap B(z,r)
\]
is contained in the $\lesssim r^2$ neighbourhood of a $k(z)$-dimensional affine subspace $V'$ containing $z$.  Therefore, we may find a cuboid $Q$ oriented with $V$ with $(d-k(z))$ short sides of length $\approx r^2$ and $k(z)$ `long' sides of length $r$ such that
\[
\mu(Q) \approx  \mu(B(z,r)) \approx r^{2\delta-k(z)}.
\]
Moreover, 
\[
\mathcal{L}^d(Q) \approx r^{k(z)+2(d-k(z))} = r^{2d-k(z)} =\left( r^2\right)^{d-k(z)/2}
\]
and so
\[
\mu(Q) \approx \mathcal{L}^d(Q)^{\frac{2\delta-k(z)}{2d-k(z)}} 
\]
and $\mu$ is $(\gamma, \upsilon)$-flat with 
\[
\gamma = \frac{2\delta-k(z)}{2d-k(z)}
\]
and
\[
\upsilon = d-k(z)/2
\]
as required.  The estimates in (1) and (2) then follow from Theorems \ref{knapp} and \ref{thm:improving}, respectively. Turning our attention to Fourier dimension and (3),  we may apply Corollary \ref{fourier} to get
    \begin{equation} \label{kleinianfromfrost}
  \fd \mu \leq \frac{2\gamma(d-\alpha)}{1-\gamma}
  \end{equation}
  where
  \begin{equation} \label{frostkleinian}
  \alpha = \frd \mu =  \min\left\{\delta,2\delta-\max_pk(z)\right\}
  \end{equation}
  is the Frostman dimension of $\mu$. This formula for $\frd \mu$ follows from the global measure formula.  Simplifying, the right hand side of \eqref{kleinianfromfrost} gives the first estimate. The condition ensuring that $\mu$ is not Salem is found by comparing the upper bound for the Fourier dimension with $\hd \mu = \delta$. To get the final estimate we look  to apply  Theorem \ref{fourier2}.  Since the $(d-k(z))$ shortest sides of $Q$ are of length $r^2$,  provided
  \[
\upsilon \gamma = \delta-k(z)/2<d-k(z),
  \]
  that is, $\delta<d-k(z)/2$,   we can apply Theorem \ref{fourier2} to get
  \[
  \fd \mu \leq 2 \upsilon \gamma = 2\delta - k(z),
  \]
  as required. The condition ensuring that $\mu$ is not Salem is found by comparing the upper bound for the Fourier dimension with $\hd \mu = \delta$ and ensuring that $\delta$ is small enough that this bound applies. 
  \end{proof}
  
  One can see that the upper bound for the Fourier dimension in Theorem \ref{kleinian} (3) becomes very good if $\delta$ is close to
  \[
  \max_z k(z)/2.
  \]
  This is a (strict) general lower bound for $\delta$ but it is sharp and so by taking $\delta$ arbitrarily close to $\max_zk(z)/2>0$ we can exhibit examples where the Fourier dimension is arbitrarily small while the Hausdorff dimension remains bounded away from zero. Such examples are perhaps surprising given the work of Bourgain--Dyatlov \cite{bourgaindyatlov} and Li--Naud--Pan \cite{naud} which established that the Fourier dimension is strictly positive in rather general situations.
  
   The estimate in the case $p=2$ in Theorem \ref{kleinian} (1)  is always a non-trivial estimate whenever $\delta<d$ in which case the lower bound for $q$ is $>2$.  Moreover, it is always better than the bound 
\[
q \geq \frac{2d}{\frd \mu}
\]
coming from Theorem \ref{basic}. This can be checked by using the formula for the Frostman dimension given in \eqref{frostkleinian} above.

  Theorem \ref{kleinian} gives simple to check conditions which ensure the Patterson--Sullivan measure $\mu$ is not a Salem measure.  We gave two conditions and neither is generally better than the other.  As an example of where the first condition is better (and the second condition does not apply at all), we state a corollary in the case where there is a parabolic point of maximal rank.
  
\begin{cor}
Let $\mu$ be the \emph{Patterson--Sullivan} measure supported on a non-elementary geometrically finite Kleinian group acting on $(d+1)$-dimensional hyperbolic space and suppose there exists a parabolic point of maximal rank, that is, of rank $d$.  Then 
  \[
   \fd \mu \leq  4\delta - 2d 
  \]
and, if $\delta<2d/3$, then $\mu$ is not Salem. Note that in this case we always have $\delta>d/2$ but any value in the range $\delta \in (d/2,d)$ is possible.
\end{cor}

\section*{Acknowledgements}

I am grateful to Ana de Orellana for many useful conversations about Fourier restriction and Knapp examples,  to Lijian Yang for helpful conversations about Fourier decay, and to Jeff Hicks for helpful conversations about differential geometry. I am also grateful to Amlan Banaji and Jonathan Hickman for helpful comments.

\end{document}